\documentclass[11pt]{amsart} 
\usepackage{fullpage, amsmath, amssymb, mathrsfs}
\usepackage[all]{xypic}

\title{Invariants of Hamiltonian flow on locally complete
  intersections} 
\author{Pavel Etingof and Travis Schedler}
\date{2014}

\usepackage{url} \usepackage{graphicx}

\numberwithin{equation}{section}

\theoremstyle{definition}
\newtheorem{theorem}[equation]{Theorem}
\newtheorem{lemma}[equation]{Lemma}
\newtheorem{proposition}[equation]{Proposition}
\newtheorem{corollary}[equation]{Corollary}

\newtheorem{definition}[equation]{Definition}

\newtheorem{example}[equation]{Example}
\newtheorem{question}[equation]{Question}
\newtheorem{conjecture}[equation]{Conjecture}
\newtheorem{remark}[equation]{Remark}

\newcommand{\sing}{\mathrm{sing}}
\newcommand{\Xsing}{X^{\sing}}
\newcommand{\Xsmth}{X^{\smth}}
\newcommand{\conv}{\text{conv}}
\newcommand{\cone}{\operatorname{cone}}
\newcommand{\smth}{\mathrm{smth}}
\newcommand{\tpl}{\operatorname{\mathsf{top}}}
\newcommand{\sd}{\operatorname{\mathsf{d}}}
\newcommand{\sdf}{\sd\!}

\newcommand{\Eu}{\operatorname{Eu}}

\newcommand{\Id}{\operatorname{Id}}

\newcommand{\ad}{\operatorname{ad}}

\newcommand{\HP}{\mathsf{HP}}
\newcommand{\IC}{\operatorname{IC}}
\newcommand{\bR}{\mathbf{R}}

\newcommand{\bZ}{\mathbf{Z}}

\newcommand{\bP}{\mathbf{P}}
\newcommand{\iso}{{\;\stackrel{_\sim}{\to}\;}}
\newcommand{\bC}{\mathbf{C}}

\newcommand{\caH}{\mathcal{H}}
\newcommand{\cO}{\mathcal{O}}

\newcommand{\caD}{\mathcal{D}}

\newcommand{\tor}{{\mathrm{tor}}}
\newcommand{\Hom}{\operatorname{Hom}}
\newcommand{\End}{\operatorname{End}}
\newcommand{\RHom}{\operatorname{RHom}}

\newcommand{\Ext}{\operatorname{Ext}}
\newcommand{\pr}{\operatorname{pr}}

\newcommand{\Spec}{\operatorname{\mathsf{Spec}}}
\newcommand{\Spf}{\operatorname{\mathsf{Spf}}}

\newcommand{\onto}{\twoheadrightarrow}
\newcommand{\into}{\hookrightarrow}

\newcommand{\bA}{\mathbf{A}}

\begin{document}


\begin{abstract}
  We consider the Hamiltonian flow on complex complete intersection
  surfaces with isolated singularities, equipped with the Jacobian
  Poisson structure. More generally we consider complete intersections
  of arbitrary dimension equipped with Hamiltonian flow with respect
  to the natural top polyvector field, which one should view as a
  degenerate Calabi-Yau structure.

  Our main result computes the coinvariants of functions under the
  Hamiltonian flow. In the surface case this is the zeroth Poisson
  homology, and our result generalizes those of Greuel, Alev and
  Lambre, and the authors in the quasihomogeneous and formal cases.
  Its dimension is the sum of the dimension of the top cohomology and
  the sum of the Milnor numbers of the singularities.  In other words,
  this equals the dimension of the top cohomology of a smoothing of
  the variety.

  More generally, we compute the derived coinvariants, which replaces
  the top cohomology by all of the cohomology. Still more generally we
  compute the $\caD$-module which represents all invariants under
  Hamiltonian flow, which is a nontrivial extension (on both sides) of
  the intersection cohomology $\caD$-module, which is maximal on the
  bottom but not on the top.  For cones over smooth curves of genus
  $g$, the extension on the top is the holomorphic half of the maximal
  extension.
\end{abstract}
\maketitle

\section{Introduction}
We explain how to recover the top cohomology and Milnor numbers from
complete intersection surfaces with isolated singularities via their
Poisson structure. Namely, we prove that the zeroth Poisson homology
of these surfaces is isomorphic to the direct sum of the top
cohomology and vector spaces of dimension equal to the Milnor numbers
of the singularities.  We will generalize this result in two
directions: to higher-dimensional complete intersections (replacing
zeroth Poisson homology by coinvariants of Hamiltonian flow) and to
all of the topological cohomology of the singular variety (by
replacing coinvariants by derived coinvariants).  We expose these
results as a series of generalizations. Then, in \S \ref{s:dmr}, we
explain and strengthen these results using $\caD$-modules, which is
also the key to their proof.  In \S \ref{s:mmax} we state our other
main results, which are on the structure of the $\caD$-module used in
\S \ref{s:dmr}.

The proofs of all results in this and the next section will be
postponed to \S \ref{s:mr-dmr-pf}, but see also \S
\ref{ss:mr-proof-outline} below for an outline of the proofs.
  \label{s:mr}
  \subsection{Complete intersection surfaces with isolated
    singularities}
  We work in the contexts of complex algebraic or complex analytic
  varieties (when we say ``affine,'' we mean closed subvarieties of
  $\bA^n$ for some $n$).  In this subsection we will restrict to the
  algebraic setting.

  Let $X$ be a surface which is an algebraic complete intersection in
  $Y := \bA^{k+2}$.  Then $X$ is Poisson, equipped with the Jacobian
  Poisson structure, defined as follows.  Suppose $X$ is cut out by
  functions $f_1, \ldots, f_k$, with $\dim Y = k+2$. Let $\Xi_Y
  := \partial_{x_1} \wedge \cdots \wedge \partial_{x_{k+2}} \in
  \wedge^{k+2} T_Y$ be the top polyvector field on $Y$.  Then we can
  define a bivector on $Y$ by the formula
\[
\pi := i_{\Xi_Y} (df_1 \wedge \cdots \wedge df_k)
\]
The bivector $\pi$ induces the following skew-symmetric bracket
(satisfying the Leibniz rule):
\[
\{g,h\} := i_\pi (\sdf g \wedge \sdf h) = i_{\Xi_Y}(\sdf f_1 \wedge
\cdots \wedge \sdf f_k \wedge \sdf g \wedge \sdf h).
\]
It is elementary, but important, to observe the following:
\begin{enumerate}
\item $\pi$ is Poisson, i.e., $\{-,-\}$ satisfies the Jacobi identity;
\item The functions $f_1, \ldots, f_k$ are Poisson central (otherwise
  known as Casimirs), so that $\pi$ and $\{-,-\}$ descend to $\cO_X =
  \cO_Y / (f_1, \ldots, f_k)$;
\item The resulting Poisson structure is nondegenerate on the smooth
  locus, call it $X^{\smth}$, of $X$, i.e., $X^{\smth}$ is a (not
  necessarily affine) symplectic surface.
\end{enumerate}
Briefly, the first fact holds because $\{-,-\}$ has generic rank two;
in fact the leaves of the Hamiltonian vector fields $\xi_f := \{f,-\}$
are the level surfaces of $f_1, \ldots, f_k$, and the restriction of
$\pi$ to each such level surface is Poisson (as is true for every
bivector on a surface).  The second fact follows immediately from the
definition.  The third fact follows because the smooth locus of $X$
and the nondegeneracy locus of $\pi$ are both the locus on $X$ where
$df_1 \wedge \cdots \wedge df_k$ does not vanish.

We will restrict our attention to the case where $X$ has only isolated
singularities. Since $X$ is a complete intersection (in particular,
Cohen-Macaulay), it is therefore normal, i.e., global functions on
$X^\smth$ equal global functions on $X$.  Thus, in this case, $\pi$ is
the \emph{unique} Poisson structure which is nondegenerate on the
smooth locus, up to scaling by invertible functions $\cO_X^\times$.
In particular, up to scaling, $\pi$ is independent of the choice of
embedding $X \into Y$.  In other words, given an arbitrary normal
surface, there might not exist such a $\pi$, but when it exists, it is
unique up to scaling, and in the complete intersection case it always
exists.

Our main result computes the zeroth Poisson homology of the Poisson
algebra $\cO_X$, i.e., $\HP_0(\cO_X) := \cO_X/\{\cO_X, \cO_X\}$.  Let
$H_{\tpl}^\bullet(X)$ denote the topological cohomology of $X$ under
the complex topology. Let $X^{\sing} \subseteq X$ be the singular
locus of $X$, which is finite by assumption. For every $s \in
X^{\sing}$, let $\mu_s$ be the Milnor number of the singularity at $s
\in X^{\sing}$.
\begin{theorem}\label{t:mt-h0-2d}
  $\HP_0(\cO_X) \cong H_{\tpl}^2(X) \oplus \bigoplus_{s \in X^{\sing}}
  \bC^{\mu_s}$.
\end{theorem}
Note that $\HP_0(\cO_X)$ also is independent (up to canonical
isomorphism) of the choice of $\pi$ up to scaling by invertible
functions, and the Milnor numbers $\mu_s$ are as well; thus, all
objects in the theorem are intrinsic to $X$.

\begin{remark}\label{r:t-mt-h0-2d-qh}
  In the local case where $X$ is either a formal or local analytic
  neighborhood of a singular point $s$, the RHS of the theorem reduces
  to $\bC^{\mu_s}$, and the theorem was proved in \cite[Corollary
  5.9]{ES-dmlv} (in the formal setting, but the analytic case also
  follows from the material of \cite[\S 5.1, 5.2]{ES-dmlv}). This
  relied primarily on \cite[Proposition 5.7.(iii)]{Gre-GMZ}.

  In the case when $\cO_X$ is a quasihomogeneous complete intersection
  with respect to some weights on $\bA^n$, the RHS of Theorem
  \ref{t:mt-h0-2d} similarly reduces to $\bC^{\mu}$ for $\mu=\mu_0$
  the Milnor number of the origin.  Again in this case, the theorem
  was proved in \cite[Theorem 5.21]{ES-dmlv}.  Moreover, one can
  describe the weight grading on $\HP_0(\cO_X)$: it is isomorphic, as
  a graded vector space, to the singularity ring, $\cO_X /
  (\frac{\partial(f_1,\ldots,f_k)}{\partial(x_{i_1},\ldots,x_{i_k})})$,
  where
  $\frac{\partial(f_1,\ldots,f_k)}{\partial(x_{i_1},\ldots,x_{i_k})}$
  is the determinant of the matrix of partial derivatives
  $\frac{\partial f_p}{\partial x_{i_q}}$ for $1 \leq p,q \leq k$.
\end{remark}

\begin{remark}\label{r:ci-bundles}
  It was not essential above that $Y = \bA^{k+2}$. Indeed, we could
  let $Y$ be an arbitrary affine Calabi-Yau variety of dimension $k+2$
  and $\Xi_Y \in \wedge^{k+2} T_Y$ a nonvanishing polyvector field
  (here and elsewhere, by Calabi-Yau, we mean only that there exists a
  nonvanishing global (algebraic) volume form, and do not require any
  compactness condition; cf.~\S \ref{ss:gen-hdv} below and its
  footnote).  The construction of the Jacobian Poisson bivector,
  nondegenerate on $X^{\smth}$, in fact does not even require the
  affine or algebraic conditions (although our theorem does).  More
  generally, we could let $Y$ be any (not necessarily affine) smooth
  analytic variety, and take the complete intersection of sections
  $f_1 \in \mathcal{L}_1, \ldots, f_k \in \mathcal{L}_k$ of line
  bundles $\mathcal{L}_1, \ldots \mathcal{L}_k$ whose tensor product
  is isomorphic to $\wedge^{\dim Y} T_Y$. Still more generally, we
  could let $f_1, \ldots, f_k$ be sections of a vector bundle
  $\mathcal{V}$ of rank $k$ whose top exterior power is isomorphic to
  $\wedge^{\dim Y} T_Y$.  Then $\pi$ still is constructed as above
  (one needs to pick a connection locally, but the result is
  independent of the choice of connection), and in the case $X$ is
  affine, the theorem applies with the same proof.
\end{remark}

\subsubsection{Restatement in terms of the smoothing of
  $X$}\label{sss:smooth}
Since $X$ is a complete intersection, it is equipped with a smoothing,
$\pi: \mathcal{X} \to \bA^1$, so that $X = \pi^{-1}(0)$ and $X_t :=
\pi^{-1}(t)$ is smooth for generic $t$.  Let us denote the Betti
numbers by $h^i(Z) := \dim H_{\tpl}^i(Z)$ for any topological space
$Z$. Then it is well-known that $h^2(X_t) = h^2(X) + \sum_s \mu_s$ for
generic $t$: this is a consequence of the fact \cite{Mil-spch,
  Ham-ltekr} that, for every $s \in X^{\sing}$ and $0 < |t| \ll 1$,
the intersection of $X_t$ with a small ball about $s \in \mathcal{X}$
is homotopic to a bouquet of $\mu_s$ $2$-spheres. We conclude that,
for generic $t$, $\HP_0(\cO_{X_t}) \cong \HP_0(\cO_X)$.  In other
words, the theorem is equivalent to the following result.
\begin{corollary} \label{c:mt-h0-2d-fl} The sheaf $\HP_0(X_t)$ on the
  line $\bA^1$ is a vector bundle near $t=0$ of rank $h^2(X) + \sum_{s
    \in X^{\sing}} \mu_s$. The generic fiber is $H_{\tpl}^2(X_t)$.
\end{corollary}
Note that being a vector bundle is the same as having fibers of
constant dimension.

\subsection{Generalization to locally complete intersections}
\label{ss:gen-lci}
More generally, we can let $X$ be an arbitrary affine surface with
isolated singularities at $X^{\sing} \subseteq X$ which is, near each
$s \in X^{\sing}$, Zariski locally an algebraic complete intersection
(in an affine space).  Still more generally, we could assume only that
$X$ is analytically locally an analytic complete intersection (in a
polydisc).

Moreover we assume that $X$ is equipped with a Poisson structure which
vanishes only at the singular locus. Then we again prove Theorem
\ref{t:mt-h0-2d}, stated in exactly the same way.  As before, when
this Poisson structure exists, it is unique up to multiplication by a
nonvanishing function. Moreover the choice of such Poisson structure
does not affect the statement of the theorem, as remarked earlier.  So
the Poisson structure is a condition on $X$, not a structure.

Moreover, provided a smoothing $X_t$ exists, Corollary
\ref{c:mt-h0-2d-fl} follows.  Note that we do not need a global
deformation over $\bA^1$; we could work with smoothing over a formal
disc $\Spf \bC[\![t]\!]$.

\begin{remark}\label{r:flconn}
  Our results generalize to the case where $X$ need not admit a
  Poisson structure nonvanishing on $X^\smth$, but admits a flat
  connection on $T_X^{2} := \Hom(\wedge^{2} \Omega_X^1, \cO_X)$,
  considered in \cite[\S 3.5]{ES-dmlv} (we continue to require that
  $X$ is analytically locally a complete intersection in a polydisc
  and has isolated singularities).  For example, when $X$ is itself
  smooth, then in the analytic setting this says that the universal
  cover of $X$ has trivial canonical bundle and admits a flat
  connection invariant under deck transformations.

  The reason why a flat connection on $T_X^2$ suffices is because our
  arguments only require the notion of Hamiltonian vector fields, not
  that of Poisson bivectors. Indeed, if $H(X)$ is the Lie algebra of
  Hamiltonian vector fields, then $\HP_0(\cO_X) = (\cO_X)_{H(X)}$, the
  coinvariants of $\cO_X$ under $H(X)$.

  Now, given a flat connection $\nabla: T_X^2 \to T_X^2 \otimes
  \Omega_X^1$, for every local section $\eta \in \Gamma(U,T_X^{2})$,
  we can define the Hamiltonian vector field on $U$, $\xi_{\eta, f} :=
  i_\eta(df) + f \text{ctr}(\nabla(\eta))$, where if $\nabla(\eta) =
  \sum_j \theta_j \otimes \alpha_j$, then $\text{ctr}(\nabla(\eta)) =
  \sum_j i_{\theta_j} (\alpha_j)$.  Using this, we define a presheaf
  of Hamiltonian vector fields, and we can take its
  sheafification. Then global sections of the latter form a Lie
  algebra, call it $LH(X)$, and we can consider $(\cO_X)_{LH(X)}$.

  With this definition, $\xi_{g\eta, f} = \xi_{\eta, fg}$. Therefore,
  if $\eta \in \Gamma(U,T_X^{2})$ is nonvanishing, all Hamiltonian
  vector fields are of the form $\xi_{g \eta, f} = \xi_{\eta,fg}$ for
  some $f,g \in \cO_X$.  Moreover, if $\eta$ is a flat section of
  $\nabla$ on $U \subseteq X$,
 then $\xi_{\eta, f} = i_\eta(df)$.  Therefore, when there is locally
 a flat nonvanishing section $\pi$ of $T_X^2$, then Hamiltonian vector
 fields are locally the same as those associated to the Poisson bivector
 $\pi$. 

 Conversely, if $X$ is Poisson (and nondegenerate on $X^{\smth}$),
 then $T_X^2$ has a flat connection uniquely determined such that the
 Poisson bivector is a flat section (since $X$ has only isolated
 singularities).  Then the Lie algebra $LH(X)$ defined above is the
 Lie algebra of vector fields which are locally Hamiltonian with
 respect to $\pi$.
 
 Note that, for $X$ Poisson (and nondegenerate on $X^{\smth}$), the
 Lie algebra $LH(X)$ may be larger than $H(X)$: elements of $LH(X)$
 are of the form $\eta_\alpha := i_\pi \alpha$ where $\alpha$ is only
 locally exact
 (whereas $H(X)$ requires $\alpha$ to be globally exact, so that there
 exists a Hamiltonian function $f$ with $\eta_\alpha = \eta_{df} =
 \xi_f$).  However all of the arguments of this paper, which are of a
 local nature, are independent of this distinction, so under the
 hypotheses of this subsection, $(\cO_X)_{LH(X)} = (\cO_X)_{H(X)} =
 \HP_0(\cO_X)$, and therefore the main results of this paper all
 extend to the setting of this remark.
\end{remark}

\subsection{Generalization to higher dimensional
  varieties}\label{ss:gen-hdv}

We generalize this result to the case where $\dim X \geq 2$, replacing
$\HP_0(\cO_X)$ by the coinvariants $(\cO_X)_{H(X)}$ of functions
$\cO_X$ under the Lie algebra of Hamiltonian vector fields defined in
\cite[\S 3.4]{ES-dmlv}, which we recall below.

Briefly, when $X \subseteq Y := \bA^{n+k}$ is a complete intersection
of dimension $n$ by functions $f_1, \ldots, f_k$, then $X$ is equipped
with a canonical polyvector field,
\[
\Xi_X := i_{\Xi_Y}(\sdf f_1 \wedge \cdots \wedge \sdf f_k)|_X.
\]
Namely, it is easy to check that $i_{\Xi_Y}(\sdf f_1 \wedge \cdots
\wedge \sdf f_k)\in \wedge^{\dim X} T_Y$ is parallel to $X$ and thus
restricts to a top polyvector field $\Xi_X$ on $X$. Moreover, as
before, $\Xi_X$ is nondegenerate on $X^{\smth}$, i.e., $(X^{\smth},
\Xi_X^{-1})$ is Calabi-Yau.  Here and henceforth, by Calabi-Yau we
only mean a smooth variety together with a nonvanishing global volume
form (we do not require any compactness condition).\footnote{We remark
  that every smooth variety is locally Calabi-Yau in our sense, e.g.,
  by taking any local top differential form and restricting to its
  nonvanishing locus.  This is probably why in some places the
  Calabi-Yau condition is accompanied by a compactness condition.
  This also explains the perhaps initially surprising fact that, for
  instance, smooth hypersurfaces in affine space of arbitrary degree
  are Calabi-Yau in our sense (as opposed to the case of projective
  space, where only hypersurfaces of degree one more than the
  dimension of the projective space can be Calabi-Yau); also affine
  space itself is Calabi-Yau in our sense. Viewed differently, our
  definition is a higher-dimensional generalization of symplectic
  surfaces, where no compactness condition is generally imposed.}
\begin{remark}\label{r:ci-bundles2}
  As in Remark \ref{r:ci-bundles}, we could more generally take $Y$
  Calabi-Yau of dimension $n+k$ with $\Xi_Y \in \wedge^{n+k} T_Y$
  nonvanishing, or an arbitrary smooth $Y$ and the complete
  intersection of sections of a vector bundle whose top exterior power
  is $\wedge^{\dim Y} T_Y$.
\end{remark}
In the case that $X$ has isolated singularities (in fact, even if it
has singularities of codimension at least two), then once again $\Xi_X$
is the \emph{unique} top polyvector field on $X$ which is nonvanishing
on $X^{\smth}$, up to scaling by $\cO_X^\times$.

As observed in \cite[\S 3.4]{ES-dmlv}, one can define Hamiltonian
vector fields on $(X, \Xi_X)$ by contracting $\Xi_X$ with exact $(\dim
X - 1)$-forms.  Then Theorem \ref{t:mt-h0-2d} generalizes to
\begin{theorem}\label{t:mt-h0-hd} For affine $X$ as above,
  $(\cO_X)_{H(X)} \cong H_{\tpl}^{\dim X}(X) \oplus \bigoplus_{s \in X^{\sing}}
  \bC^{\mu_s}$.
\end{theorem}
More generally, as in the surface case, we prove this theorem when $X$
is an arbitrary affine locally complete intersection of dimension
$\geq 2$ with isolated singularities equipped with a polyvector field
$\Xi_X$ which is nondegenerate on the smooth locus, i.e., such that
$X^{\smth}$ is Calabi-Yau.
Here as before ``locally complete intersection'' means (in the
analytic context) that every singular point $s \in X^\sing$ has an
analytic neighborhood which is an analytic complete intersection in a
polydisc.  When $X$ is algebraic, this includes the case where $s$ has
a Zariski neighborhood which is a complete intersection.  We remark,
as before, that since $\Xi_X$ is unique up to multiplication by a
nonvanishing function, it is clear that $(\cO_X)_{H(X)}$ does not
depend on the choice of $\Xi_X$ up to isomorphism.

\begin{remark}\label{r:t-mt-h0-hd-qh} As in Remark
  \ref{r:t-mt-h0-2d-qh}, in the
  case that $X$ is a formal or local analytic neighborhood of a
  singular point $s$, the theorem was proved in \cite[Corollary
  5.9]{ES-dmlv}, using \cite[Proposition
  5.7.(iii)]{Gre-GMZ}. Similarly, when $X$ is conical, i.e., it is
  affine and admits a contracting $\bC^\times$ action, and $\Xi_X$ is
  assumed to be homogeneous of some weight, the result is a
  straightforward generalization of \cite[Theorem 5.21]{ES-dmlv},
  where the graded vector space structure is proved to identify with
  that of the singularity ring.
\end{remark}

Parallel to Corollary \ref{c:mt-h0-2d-fl}, we have the result
\begin{corollary}\label{c:mt-h0-hd-fl}
  Suppose that $X_t$ is a smoothing of $X$, i.e., $\pi: \mathcal{X}
  \to \bA^1$ is a family with $X_t := \pi^{-1}(t)$ generically smooth
  and $X = X_0$. Then, the sheaf $(\cO_{X_t})_{H(X_t)}$ on $\bA^1$ is
  a vector bundle near $t=0$ of rank $h^{\dim X}(X) + \sum_{s \in
    X^{\sing}} \mu_s$.  The generic fiber is the top cohomology
  $H^{\dim X}_{\tpl}(X_t)$.
\end{corollary}
Provided a smoothing exists, this is equivalent to the
theorem. However, since $X$ is only locally a complete intersection,
we only know that there exists a smoothing of a neighborhood of each
singular point.  Note that we could alternatively work with a
smoothing over a formal disc $\Spf \bC[\![t]\!]$, and the corollary
extends to this case.

\begin{remark} \label{r:flconn2} Parallel to Remark \ref{r:flconn},
  all of our main results generalize to the case where $X$ is not
  necessarily equipped with a top polyvector field $\Xi_X$
  nonvanishing on $X^\smth$, but is instead equipped with a flat
  connection $\nabla$ on $T_X^{\dim X} := \Hom(\wedge^{\dim
    X}\Omega_X^1, \cO_X)$, cf.~\cite[\S 3.5]{ES-dmlv}, in addition to
  being analytically locally a complete intersection in a polydisc and
  having isolated singularities.  This is enough to define, as before,
  the Lie algebra $LH(X)$ of vector fields which are locally
  Hamiltonian on $X$, and the aforementioned results are all valid
  replacing $(\cO_X)_{H(X)}$ by $(\cO_X)_{LH(X)}$.  In the case where
  $\nabla$ actually admits a flat section $\Xi_X$, then as in Remark
  \ref{r:flconn}, $H(X)$ could be a proper Lie subalgebra of $LH(X)$,
  the latter which is the contraction of $\Xi_X$ with locally exact
  (rather than exact) $(\dim X-1)$-forms.  Nonetheless, as before, one has
  $(\cO_X)_{LH(X)} \cong (\cO_X)_{H(X)}$.
\end{remark}
\subsection{A question on symmetric powers}
In \cite{ESsym}, in the case that $X$ is a conical hypersurface in
$\bC^3$ (so that $X^\sing = \{0\}$ is the vertex), we computed the
zeroth Poisson homology of arbitrary symmetric powers of $X$.

Let $\lambda \vdash n$ denote that $\lambda$ is a partition of $n$,
which we write as $\lambda = (\lambda_1, \ldots, \lambda_m)$ for
$\lambda_1 \geq \lambda_2 \geq \cdots \geq \lambda_m$ and $\lambda_1 +
\cdots + \lambda_m = n$. In this case, define $|\lambda| := m$.  For
each such partition, let $S_\lambda < S_m$ denote the group of
permutations preserving the partition.  Explicitly, $S_{\lambda} =
S_{r_1} \times \cdots \times S_{r_k}$ where, for all $j$,
\[
\lambda_{r_1+\cdots+r_j} > \lambda_{r_1+\cdots+r_j+1} = \lambda_{r_1 +
  \cdots + r_j + 2} = \cdots = \lambda_{r_1 + \cdots + r_j + r_{j+1}}.
\]
Define also, for an arbitrary vector space $V$,
\[
S^\lambda V := (V^{\otimes |\lambda|})^{S_\lambda}.
\]
In \cite{ESsym}, we proved that, for $X$ a conical symplectic surface in $\bC^3$
with an isolated singularity,
\begin{equation}\label{e:symm}
  \HP_0(\cO_{S^n X}) \cong \bigoplus_{\lambda \vdash n} S^\lambda \HP_0(\cO_X),
\end{equation}
and, by \cite{AL}, $\HP_0(\cO_X) \cong \bC^{\mu_s}$ in this case.
By definition, $\HP_0(\cO_{S^n X}) = (\cO_{S^n X})_{H(S^n X)}$.  As
observed in \cite[Proposition 1.4.1]{ESsym}, this also equals
$(\cO_{S^n X})_{H(X)}$, where $H(X)$ acts on $S^n X$ by the diagonal
action. This allows us to state a generalization of the above equality
when $X$ is an arbitrary complete intersection (and not just for
surfaces):
\begin{question}
  Does the analogue of \eqref{e:symm} hold when $X$ is replaced by an
  arbitrary locally complete intersection with isolated singularities?
  That is, does one have
\begin{equation}\label{e:symm2}
  (\cO_{S^n X})_{H(X)} \cong \bigoplus_{\lambda \vdash n} S^\lambda\bigl( (\cO_X)_{H(X)} \bigr)?
\end{equation}
\end{question}
By Theorem \ref{t:mt-h0-hd}, \eqref{e:symm2} can be rewritten as
\begin{equation}
  (\cO_{S^n X})_{H(X)}  \cong \bigoplus_{\lambda \vdash n}  S^\lambda \bigl(
  H_{\tpl}^{\dim X}(X) \oplus \bigoplus_{s \in X^{\sing}} \bC^{\mu_s}\bigr).
\end{equation}
\subsection{Acknowledgements}
The first author's work was partially supported by the NSF grant
DMS-1000113. The second author was supported by a five-year fellowship
from the American Institute of Mathematics, and the ARRA-funded NSF
grant DMS-0900233. Part of this work was conducted while the second
author was in residence at the Mathematical Science Research Institute
(MSRI) in Berkeley, California, during 2013, supported by NSF Grant
No. 0932078 000.

\section{Derived coinvariants, all de Rham cohomology, and
  $\caD$-modules}\label{s:dmr}
\subsection{$\caD$-module generalization for surfaces}
In order to prove the above results we actually prove a much more
general result, which also computes the \emph{derived coinvariants} of
$\cO_X$ under Hamiltonian flow.  To formulate these, we use
$\caD$-modules. These also have the advantage of being local, so that
the final statements we prove for complete intersections will
immediately imply the generalization to locally complete
intersections.

\subsubsection{The $\caD$-module $M(X)$}
We need to recall the essential construction of \cite{ESdm}.  Let $X$
be an affine Poisson (complex algebraic or complex analytic) variety
and $H(X)$ its Lie algebra of Hamiltonian vector fields.  Let $M(X)$
be the right $\caD$-module defined in \cite{ESdm} representing
invariants under Hamiltonian flow. We briefly recall its definition:
\[
M(X) := H(X) \cdot \caD_X \setminus \caD_X,
\]
where $\caD_X$ is the right $\caD$-module on $X$ representing global
sections, i.e., $\Gamma(X,M) = \Hom(\caD_X, M)$ for all right
$\caD$-modules; this carries a canonical left action by vector fields
on $X$, and $H(X) \cdot \caD_X$ is then a right $\caD$-submodule of
$\caD_X$.

Then, by definition, one has \cite[Proposition 2.13]{ESdm}:
\begin{equation}
\HP_0(\cO_X) \cong \pi_0 M(X),
\end{equation}
where $\pi: X \to \Spec \bC$ is the projection to point and $\pi_0
M(X) = H_0 (\pi_* M(X))$ is the underived pushforward.

Recall the definition of higher Poisson-de Rham homology (for an
arbitrary Poisson variety $X$):
\begin{definition}\cite[Remark 2.27]{ESdm}
  $\HP^{DR}_*(X) := \pi_* M(X)$.
\end{definition}
(As explained in \cite[Remark 2.27]{ESdm}, for $X$ smooth symplectic, this
coincides with the de Rham cohomology of $X$: $\HP^{DR}_*(X) \cong
H_{\tpl}^{\dim X - *}(X)$.)
\begin{remark}
  As explained in \cite{ESdm}, all of the above generalizes to
  arbitrary (not necessarily affine) algebraic or analytic varieties
  $X$, when one replaces $H(X)$ by the \emph{presheaf} $\caH(X)$ of
  Hamiltonian vector fields. We remark that $\caH(X)$ is \emph{not}
  necessarily a sheaf: see Remark \ref{r:flconn} above and
  \cite[Remark 4.5]{ES-dmlv}; this does not cause any trouble in the
  definition of $M(X)$, as explained there.  The reason, briefly, is
  that $\caH(X)$ is $\caD$-localizable (\cite[Corollary
  4.4]{ES-dmlv}), i.e., for every affine open $U \subseteq X$ and open
  $V \subseteq U$, one has $\caH(U)|_V \cdot \caD_V = \caH(V) \cdot
  \caD_V$.  That is, the $\caD$-module on $U$ generated by global
  Hamiltonian vector fields coincides with that generated by local
  Hamiltonian vector fields.  Therefore the quotient $M(U) := \caH(U)
  \cdot \caD_U \setminus \caD_U$ has the property $M(U)|_V = M(V)$.
\end{remark}
\begin{remark}
  In fact, the same $\caD$-localizability property holds for locally
  Hamiltonian vector fields, $\mathcal{LH}(X)$, defined by contracting
  the Poisson bivector with closed rather than exact one-forms,
  cf.~Remark \ref{r:flconn}. However, this defines the same $\caD$-module
  as $M(X)$ since locally Hamiltonian vector fields are in the
  $\cO_X$-linear span of Hamiltonian vector fields (more precisely,
  $LH(U)$ and $H(U)$ have the same $\cO$-saturations for all open
  affine $U \subseteq X$, and hence $LH(U) \cdot \caD_U = H(U) \cdot
  \caD_U$, as explained in \cite[Proposition 3.11]{ES-dmlv}).  We work
  with Hamiltonian vector fields since it is more convenient for us
  and then the definition of $M(X)$ matches that of \cite{ESdm}.

  In the more general situation of Remark \ref{r:flconn} where one is
  equipped with a flat connection on $T_X^2$, one can similarly see
  that the presheaf $\mathcal{LH}(X)$ is $\caD$-localizable and hence
  in this setting we can also allow $X$ to be nonaffine.  All of the
  results below extend to this general setting.
\end{remark}

\subsubsection{Poisson-de Rham homology of (locally) complete
  intersection surfaces}
Now let $X$ be a surface as in \S \ref{ss:gen-lci}. As before, this
includes the cases of an algebraic complete intersection in affine
space or in a Calabi-Yau variety (or the more general example of
Remark \ref{r:ci-bundles}).  We do not need $X$ to be affine.

Then we have the following result which generalizes Theorem
\ref{t:mt-h0-2d}:
\begin{theorem}\label{t:mt-hs-2d}
  $\HP^{DR}_i(X) \cong \begin{cases}
    H_{\tpl}^{2-i}(X), & \text{if $i > 0$}, \\
    H_{\tpl}^2 \oplus \bigoplus_{s \in X^{\sing}} \bC^{\mu_s}, &
    \text{if $i=0$}.
\end{cases}$
\end{theorem}

Restricting to the case where $X$ is a complete intersection in affine
space
there always exists a smoothing $X_t$ of $X$. Then Theorem
\ref{t:mt-hs-2d} is equivalent to the following analogue of Corollary
\ref{c:mt-h0-2d-fl}.  As before, $h^i(X) := \dim H^i(X)$ denotes the
$i$-th Betti number of $X$.
\begin{corollary}\label{c:mt-hs-2d-fl}
  The sheaf of graded vector spaces $\HP^{DR}_*(X_t)$ is a vector
  bundle near $t=0$.  Its generic fiber is $H^{\dim X -
    *}_{\tpl}(X_t)$, and the Hilbert series of the fibers near $t=0$
  is $h(\HP^{DR}_*(X_t);u) = \sum_{m} h^{\dim X - m}(X) u^m + \sum_{s
    \in X^{\sing}} \mu_s$.
\end{corollary}
Note that being a graded vector bundle is merely saying that the
Hilbert series of the fibers is constant.

The same corollary holds in the general case (so $X$ is only
analytically locally a complete intersection, and not necessarily
affine), provided it is equipped with a smoothing $X_t$ (as before one
could also work with a smoothing over a formal disc).

\subsubsection{The $\caD$-module of (locally) complete intersection
  surfaces}\label{ss:dm-2d}
Let $X$ be as in the previous subsection.  Let $j: X^{\smth} \into X$
be the inclusion of the smooth locus into $X$.  Let
$\Omega_{X^{\smth}}$ be the canonical right $\caD$-module of volume
forms on $X^{\smth}$.  Then we can consider the underived pushforward
$H^0 j_! \Omega_{X^{\smth}}$. This is the maximal indecomposable
extension on the bottom (by $\caD$-modules supported on $\Xsing$) of
the intersection cohomology $\caD$-module $\IC(X) := j_{!*}
\Omega_{X^{\smth}}$, where $j_{!*}$ is the intermediate extension.
More precisely, we have the following standard result, for which we
provide a proof for the reader's convenience:
\begin{proposition}
  $N := H^0 j_! \Omega_{X^{\smth}}$ is indecomposable, and is the
  universal extension of the form
\begin{equation}\label{e:js-ext}
0 \to K \to N \to \IC(X) \to 0,
\end{equation}
with $K$ supported on $\Xsing$, in the sense that every other
extension
\begin{equation}\label{e:oth-ext}
  0 \to K' \to N' \to \IC(X) \to 0,
\end{equation}
with $K'$ supported on $\Xsing$, is the pushout of a canonical
morphism $K \to K'$.
\end{proposition}
\begin{proof}
  First of all, by adjunction, we have a canonical morphism in the
  derived category, $j_! \Omega_{\Xsmth} \to \IC(X)$, which becomes
  the identity after applying $j^!$. Since $j_! \Omega_{\Xsmth}$ is a
  complex concentrated in nonpositive (cohomological) degrees, it
  follows that, for all $\caD$-modules $L$,
\begin{equation}\label{e:h0rhom}
  H^0 \RHom(j_! \Omega_{\Xsmth}, L) \cong \Hom(H^0 j_! \Omega_{\Xsmth}, L) = \Hom(N, L).
\end{equation}
In more detail, take the exact triangle $j_! \Omega_{\Xsmth} \to H^0
j_! \Omega_{\Xsmth} \to C$, for $C = \cone(j_! \Omega_{\Xsmth} \to H^0
j_! \Omega_{\Xsmth})$. The cohomology of $C$ is concentrated in
negative degrees. Then the long exact sequence for $\Hom(-,L)$ yields,
for all $\caD$-modules $L$, the exact sequence $0 \to \Hom(H^0 j_!
\Omega_{\Xsmth}, L) \to H^0 \RHom(j_! \Omega_{\Xsmth}, L) \to 0$,
since $H^i(C) = 0$ for $i \geq 0$.
  
Thus, by \eqref{e:h0rhom}, we obtain from adjunction that
\begin{equation}\label{e:hom-nic}
\Hom(N, \IC(X)) = H^0 \RHom(j_! \Omega_{\Xsmth}, \IC(X))
\cong \Hom(\Omega_{\Xsmth}, \Omega_{\Xsmth}) = \bC.
\end{equation}
We therefore obtain the canonical extension \eqref{e:js-ext}.

To see that $N$ is indecomposable, note that, by the same computation
of \eqref{e:hom-nic}, $\Hom(N, N) \cong H^0 \RHom(j_! \Omega_{\Xsmth},
N) \cong \Hom(\Omega_{\Xsmth}, \Omega_{\Xsmth}) = \bC$.

Finally, we show that $N$ is the universal (maximal indecomposable)
extension of $\IC(X)$ by $\caD$-modules supported on $\Xsing$.  If
not, let $N'$ be this extension (it exists since $\Xsing$ is finite
and $\Ext(\IC(X), \delta_s)$ is finite-dimensional for all $s \in
\Xsing$).  By universality, the surjection $N' \to \IC(X)$ factors
through $N \to \IC(X)$, but by adjunction and \eqref{e:h0rhom},
$\Hom(N, N') = \bC$.  Thus we obtain a nonzero composition $N \to N'
\to N$, which must be a nonzero multiple of the identity since
$\Hom(N,N) = \bC$.  Since $N'$ is indecomposable, this implies that
$N=N'$, a contradiction. \qedhere
\end{proof}

We can make \eqref{e:js-ext} more explicit as follows. For every $s
\in X^{\sing}$, let $\delta_s$ denote the $\delta$-function
$\caD$-module at $s$. Then,
\begin{equation} \label{e:k-ext} K \cong \bigoplus_{s \in X^{\sing}}
  \Ext^1(\IC(X), \delta_s)^* \otimes \delta_s.
\end{equation}
This can be written in terms of topological cohomology: for every $s
\in X^{\sing}$, let $U_s$ be a contractible neighborhood of $s$ (whose
existence was proved in \cite{Gil-eavlc}, cf.~also
\cite[2.10]{Mil-spch}), disjoint from $X^\sing \setminus \{s\}$.  Then
by a straightforward generalization of \cite[Lemma 4.3]{ESdm},
\begin{equation}\label{e:ic-ds}
  \Ext^1(\IC(X), \delta_s) \cong H^1(U_s
  \setminus \{s\}).
\end{equation} Thus \eqref{e:k-ext} can be alternatively written as
\begin{equation}\label{e:k-h1}
  K \cong \bigoplus_{s \in X^{\sing}} H^1(U_s \setminus \{s\})  \otimes \delta_s.
\end{equation}

Then Theorems \ref{t:mt-hs-2d} and \ref{t:mt-h0-2d} follow from the
following result.
\begin{theorem}\label{t:mt-dm-2d} 
  There is an exact sequence
  \[
  0 \to H^0 j_! \Omega_{X^{\smth}} \to M(X) \onto \bigoplus_{s \in
    X^{\sing}} \bC^{\mu_s} \otimes \delta_s \to 0.
  \]
\end{theorem}
One advantage of using $\caD$-modules is that this is a local
statement: so the statement above for complete intersections now
immediately implies the generalization to the case where $X$ is only
analytically locally complete intersection as in \S \ref{ss:gen-lci}
(and not necessarily affine).  Indeed, we only need to observe that
the $\caD$-module $M(X)$ does not depend on the choice of $\Xi_X$, up
to isomorphism, since $\Xi_X$ is unique up to multiplication by an
everywhere nonvanishing function.  In more detail, if $f \in
\Gamma(X,\cO_X)$ is everywhere nonvanishing, one has the isomorphism
$M(X,\Xi_X) \to M(X, f \Xi_X)$, which on affine open subsets $U
\subseteq X$ sends the canonical generator $1 \in M(U, \Xi_X|_U)$ to
$f^{-1}$ times the canonical generator of $M(U, f \Xi_X|_U)$.

We can also prove an analogue of Corollaries \ref{c:mt-h0-2d-fl} and
\ref{c:mt-hs-2d-fl} (which we will actually prove simultaneously with
the preceding theorem).  Suppose that $X_t$ is a smoothing of $X$
(over $\bA^1$ or over a formal disc $\Spf \bC[\![t]\!]$); this always
exists if $X$ is a complete intersection.  Let $\mathcal{X}$ be the
total space of the smoothing.
\begin{theorem}\label{t:mt-dm-2d-fl}
  The sheaf over $\bA^1$ of (fiberwise) $\caD$-modules $M(X_t)$ on
  $\mathcal{X}$ is flat near $t=0 \in \bA^1$.  For generic $t \in
  \bA^1$, the fiber at $t$ is $\Omega_{X_t}$.
\end{theorem}
In the general case where $X$ is only locally a complete intersection,
we can apply the theorem to the open complex neighborhoods $U_s$ of
each $s \in X^{\sing}$ which are analytically complete intersections.
This is an important ingredient in the proof of Theorem
\ref{t:mt-dm-2d} itself.

\subsection{$\caD$-modules on higher dimensional
  varieties}\label{ss:dm-hd}
Theorems \ref{t:mt-hs-2d} and \ref{t:mt-dm-2d} also carry over
verbatim to the higher-dimensional setting with $X$ a ``degenerate
Calabi-Yau'' analytically locally complete intersection with finite
degeneracy locus.  Precisely, we let $X$ be a not necessarily affine
analytically locally complete intersection of dimension $\geq 2$
equipped with a top polyvector field $\Xi_X$ vanishing at only
finitely many points, i.e., $X$ has only isolated singularities and
$\Xi_X$ vanishes only at the singular locus.

Then, we prove Theorem \ref{t:mt-dm-2d} in this general context
(stated in exactly the same way), and this implies all of the
preceding theorems.  Note that, in this case, the maximal extension
$H^0 j_! \Omega_{X^{\smth}}$ of $IC(X)$ (by $\caD$-modules supported
on $X^\sing$) is an extension by
\begin{equation}\label{e:max-ext}
  K \cong \bigoplus_{s \in X^{\sing}} \Ext^1(IC(X), \delta_s)^* \otimes \delta_s \cong
  \bigoplus_{s \in X^{\sing}} H^{\dim X - 1}(U_s \setminus \{s\}) \otimes \delta_s,
\end{equation}
where $U_s$ is a contractible neighborhood of $s$ for every $s \in
X^{\sing}$, disjoint from $X^\sing \setminus \{s\}$.

The analogue of Theorem \ref{t:mt-hs-2d}, which follows from the
above, is as follows. Let $\pi: X \to \Spec \bC$ be the projection to
a point.
\begin{corollary}\label{c:mt-hs-hd} 
  $\pi_i M(X) \cong \begin{cases}
    H^{\dim X - i}_{\tpl}(X), & \text{if $i > 0$}, \\
    H^{\dim X}_{\tpl}(X) \oplus \bigoplus_{s \in X^{\sing}}
    \bC^{\mu_s}, & \text{if $i = 0$}.
\end{cases}$
\end{corollary}
Equivalently, if $X_t$ is a smoothing of $X$ (which exists at least
when $X$ is a global complete intersection; as before we could work
with a smoothing over a formal disc or over $\bA^1$), this becomes the
following (with $h^i(X)=\dim H^i(X)$ again the $i$-th Betti number of
$X$):
\begin{corollary}\label{c:mt-hs-hd-fl} 
  The sheaf of graded vector spaces $\pi_* M(X_t)$ on $\bA^1$ is a
  graded vector bundle near $t=0$ whose fibers have Hilbert series
  $h(\pi_* M(X_t), u) = \sum_m h^{\dim X - m}(X) u^m + \sum_{s \in
    X^{\sing}} \mu_s$.  The generic fiber is $H_{\tpl}^{\dim X -
    *}(X_t)$.
\end{corollary}
Note that being a graded vector bundle is the same as saying that the
Hilbert series of the fibers is constant.

Finally, we also prove Theorem \ref{t:mt-dm-2d-fl} in this context,
which goes through verbatim.

\subsection{Outline of the proof}\label{ss:mr-proof-outline}
We outline the proof of Theorems \ref{t:mt-dm-2d} and
\ref{t:mt-dm-2d-fl} (which imply all of the results from \S \ref{s:mr}
and \S \ref{s:dmr}) in the original situation of $X$ a complete
intersection surface in affine space. For details, see \S
\ref{s:mr-dmr-pf}.
\begin{enumerate}
\item We exploit the smoothing $X_t$, where generically $X_t$ is
  smooth and hence $M(X_t) = \Omega_{X_t}$ (by \cite[Example
  2.6]{ESdm}, since then $X_t$ is a symplectic surface).
\item We exploit the structure theory of $M(X)$: it must be an
  extension of $IC(X)$ on (possibly) both sides by delta-function
  $\caD$-modules at $X^{\sing}$.  We use that the maximal extension
  $H^0 j_!  \Omega_{X^{\smth}}$ is an extension by $K$ as in \eqref{e:k-h1}.
\item By \cite[\S 5]{ES-dmlv}, the maximal quotient supported at
  $X^{\sing}$ is $\bigoplus_{s \in X^{\sing}} \HP_0(\hat \cO_{X,s})
  \otimes \delta_s$, and $\dim \HP_0(\hat \cO_{X,s}) = \mu_s$.
\item Now, 
  $M(X_t)$ is flat
  near $t=0$ if and only if there is no torsion at $t=0$; this torsion
  would have to be supported at $X^{\sing}$.
\item \label{step-euler} We take the Euler-Poincar\'e characteristic
  of $\pi_* M(X_t)$ (for all $t$).  Since the torsion of the family
  $M(X_t)$ is a direct sum of delta-function $\caD$-modules, the
  Euler-Poincar\'e characteristic of $\pi_* M(X_t)$ can only increase
  at $t=0$, and it increases if and only if the family is not flat
  (i.e., $\pi_* M(X_t)$ is not flat viewed as a coherent sheaf on
  $\bA^1$).
\item \label{ic-step} Using the classical formula \cite[\S
  6.1]{GM-iht} for $\pi_* \IC(\bar X)$ for compact $\bar X$ (applied
  to the one-point compactification of a contractible neighborhood of
  each isolated singularity), we compute the Euler-Poincar\'e
  characteristic of $\pi_* \IC(X)$ (Proposition \ref{p:icfla}).
\item This formula shows that the Euler-Poincar\'e characteristic of
  $\pi_* M(X_t)$ cannot increase at $t=0$, so the family is flat,
  proving Theorem \ref{t:mt-dm-2d-fl}.  The formula simultaneously
  shows that $M(X)$ maximally extends $\IC(X)$ on the bottom
  (otherwise the Euler-Poincar\'e characteristic of $\pi_* M(X_t)$
  would go down at $t=0$, which is impossible), which proves Theorem
  \ref{t:mt-dm-2d}.
\end{enumerate}
The exact same outline above applies to the situation where $X$ is a
complete intersection of arbitrary dimension, except that $X_t$,
rather than being a symplectic surface, is a smooth Calabi-Yau variety
for generic $t$, so $M(X_t) \cong \Omega_{X_t}$ still holds by
\cite[Example 2.37]{ES-dmlv}.  The proof is then the same, but we have
to use in step \eqref{ic-step} the fact that, for an analytically
locally complete intersection of dimension $n$ with an isolated
singularity, the link of the singularity is $(n-2)$-connected
\cite{Mil-spch},\cite[Korollar 1.3]{Ham-ltekr}.

Once we have the results for complete intersections, as observed, we
immediately obtain them for locally complete intersections, since
Theorems \ref{t:mt-dm-2d} and \ref{t:mt-dm-2d-fl} are local. In fact,
we prove the theorem in the local setting, assuming $X$ is
contractible with a single isolated singularity, which implies the
global result.
\section{The full structure of $M(X)$ and top log differential forms
  on resolutions}\label{s:mmax}
As before we assume that $X$ is an analytically locally complete
intersection equipped with a top polyvector field $\Xi_X$ which is
nondegenerate on $X^\smth$ and the dimension of $X$ is at least two.
Theorem \ref{t:mt-dm-2d} implies that $M(X)$ is a direct sum of
delta-function $\caD$-modules together with one indecomposable
extension by $H^0 j_!  \Omega_{X^\smth}$ of delta-function
$\caD$-modules supported on $X^\sing$.  Let $M_{\max}(X)$ denote this
indecomposable extension.

To fully describe the structure of $M(X)$ and $M_{\max}(X)$, in view
of the exact sequence of Theorem \ref{t:mt-dm-2d}, we only need to
describe how much of the quotient $\bigoplus_s \bC^{\mu_s} \otimes
\delta_s$ therein splits off of $M(X)$, and how much is in the image
of $M_{\max}(X)$.  That is, it remains only to compute the quotient
$M_{\max}(X) / H^0 j_!  \Omega_{X^\smth}$, i.e., the maximal quotient
of $M_{\max}(X)$ supported on $X^\sing$.

In this section, in Theorem \ref{t:maxind-con} and Proposition
\ref{p:maxind-qt}, we do this when $X$ is conical, or more generally
locally conical near all the singular points.  For the general case,
we state Conjecture \ref{c:max-ext}, which we show follows from the
aforementioned results in the locally conical case in Proposition
\ref{p:conj-con}.  The proofs of Theorem \ref{t:maxind-con} and
Proposition \ref{p:maxind-qt} are postponed to \S \ref{s:mmax-pf}.

Here and below, a \emph{conical} variety $X$ means an affine variety
with a contracting $\bC^\times$ action to a fixed point. In the
algebraic setting this means that $\cO_X$ is nonnegatively graded (by
what we will call the \emph{weight} grading) with $\bC$ in weight
zero; in the analytic setting $\cO_X$ is a completion of the direct
sum of its homogeneous components, which are in nonnegative weights
and with $\bC$ in weight zero.  We will also use the term ``weight''
in general to refer to the grading induced by a $\bC^\times$ action. 

\subsection{The case of a $\bC^\times$-action}\label{s:con}
Suppose that $X$ is equipped with a $\bC^\times$ action, e.g., $X$ is
conical (and hence affine).
Moreover, assume that $\Xi_X$ is homogeneous of some weight $d$. Then,
the inverse volume form on $X^\smth$ is homogeneous of weight
$-d$. Next, $j^* M(X) = M_{X^\smth} \cong \Omega_{X^\smth}$, and the
 volume form is a distributional solution of $j^* M(X)$. Therefore,
for every affine open subset $U \subseteq X^\smth$, the canonical
generator $1 \in M_{U}$ is annihilated by the operator $\Eu - d$ (as
the right action of $\Eu$ on a distribution $\phi$ of weight $m$ is $\phi
\cdot \Eu = -m \phi$; alternatively, on volume forms $\Eu$ acts by
$-L_{\Eu}$, which gives the same answer).

Now, $M(X^{\smth})$ is weakly equivariant with respect to the
$\bC^\times$ action.  Since the canonical generator on every affine
open has weight $d$, this implies that $M(X^{\smth})$ is homogeneous
of weight $d$ as a weakly $\bC^\times$-equivariant $\caD$-module (the
difference of the two canonical actions of the Lie algebra $\bC \cdot
\Eu$ of $\bC^\times$ is the character $\Eu \mapsto d$).  Therefore,
$H^0 j_! M(X^{\smth}) \cong H^0 j_! \Omega_{X^\smth}$ is also weakly
equivariant of weight $d$, and this is the structure such that the
canonical morphism $H^0 j_! M(X^\smth) \to M(X)$ is a morphism of weakly
equivariant $\caD$-modules.

Decompose $M(X)$ as $M(X) = \bigoplus_{m \in M} M(X)_m$, with $M(X)_m$
the part of weight $m$. The morphism $H^0 j_! M(X^\smth) \to M(X)$
lands entirely in weight $d$.  Also, $j^* M(X)_m = 0$ for $m \neq d$
(which also follows directly because the canonical morphism $j^*
M(X)_m \to M(X^\smth)_m = 0$ is zero). So, for $m \neq d$, $M(X)_m$ is
a direct sum of delta function $\caD$-modules on points of the finite
set $X^\sing$, whereas at $m=d$, the extension of Theorem
\ref{t:mt-dm-2d} splits as follows:

\begin{equation}\label{e:mxd-ext}
0 \to H^0 j_! \Omega_{X^\smth} \to M(X)_{d} \to \bigoplus_{s \in X^\sing} E_s \to 0,
\end{equation}
where $E_s$ is the weight $d$ component of $\bC^{\mu_s} \otimes
\delta_s$ in Theorem \ref{t:mt-dm-2d}.
\begin{theorem}\label{t:maxind-con}
  Suppose that $s \in X^\sing$ is an attracting or repelling fixed
  point of the $\bC^\times$-action.  Let $Y := X^\smth \cup \{s\}$,
  and $j_Y: Y^\smth=X^\smth \into Y$ be the inclusion.  Then, the
  extension
\begin{equation}\label{e:mxd-ext-1pt}
0 \to H^0 (j_{Y})_! \Omega_{Y^\smth} \to M(Y)_{d} \to E_s \to 0
\end{equation}
is indecomposable.  In particular, if all of the points $s \in
X^\sing$ are attracting or repelling fixed points (e.g., $X$ is
conical with $X^\sing = \{s\}$), then the extension \eqref{e:mxd-ext}
is indecomposable.
\end{theorem}
The theorem will be proved in the case where $X$ is a homogeneous
hypersurface in $\bC^3$ in \S \ref{s:mmax-pf-2d} below (and hence also
where a neighborhood of $s$ is a neighborhood of such a
hypersurface). The proof will then be adapted to the general case in
\S \ref{s:mmax-pf}.

We can also explicitly describe $E_s$.  In a neighborhood $U_s$ of $s
\in X^\sing$, let $U_s$ be a complete intersection in a polydisc, and
consider the formal completion $\hat \cO_{U_s,s} = \hat \cO_{X,s}$ (or one
could use the analytic or algebraic local ring).
Then $\hat \cO_{X,s}$ also has a $\bC^\times$ action and decomposes
into weight spaces, $\hat \cO_{X,s} = \prod_{m} (\hat
\cO_{X,s})_m$.
\begin{proposition}\label{p:maxind-qt}
  Suppose $s$ is an attracting or repelling fixed point. Then, there
  is a canonical isomorphism $E_s \cong (\hat \cO_{X,s})_d \otimes \delta_s$.
\end{proposition}
This proposition will also be proved in \S$\!$\S \ref{s:mmax-pf-2d}
and \ref{s:mmax-pf} below (it follows relatively easily from results
we will recall in \S \ref{ss:maxquot}).
\begin{example} If $X$ is the cone over a smooth curve of weight $d$
  in the projective plane $\bC \mathbf{P}^2$, then one may check that
  $\dim (\cO_{X})_d = g$, where $g = \frac{(d-1)(d-2)}{2}$ is the
  genus of the curve. In fact, for arbitrary conical surfaces $X$, if
  $X^\smth / \bC^\times$ is a curve of genus $g$, then it is still
  true that $\dim (\cO_X)_d = g$: this follows from Example
  \ref{ex:conj-sfc} and Propositions \ref{p:conj-rewr} and
  \ref{p:conj-con} below.  Since, in this case, $H^0 j_!
  \Omega_{X^\smth}$ is an extension of the irreducible $\caD$-module
  $\IC(X)$ by $\delta_s^{2g}$ (by \eqref{e:ic-ds}), we see that
  $M(X)_{\max}$ has a filtration with subquotients $\delta_s^g$ on the
  top, $\IC(X)$ in the middle, and $\delta_s^{2g}$ on the bottom. That
  is, the extension is maximal on the bottom, and half-maximal on the
  top.
\end{example}
\begin{remark}\label{r:us-hat}
  The fact that \eqref{e:mxd-ext-1pt} is indecomposable is a local
  statement, i.e., it can be applied to neighborhoods $U_s$ of each
  point $s$. Thus, we deduce the following statement: If $s \in X$ has
  a neighborhood $U_s$ which is isomorphic to a neighborhood $V_s$ of
  a conical variety $Y$, then $M(U_s)_{\max} \cong M(Y)|_{V_s}$ is given by an exact
  sequence, for $U_s^\circ := U_s \setminus \{s\}$,
\begin{equation}\label{e:mxd-ext-1pt-loc}
0 \to H^0 (j_{U_s^\circ})_! \Omega_{U_s} \to M(U_s)_{\max} \to (\hat \cO_{Y,s})_d \otimes \delta_s
\to 0.
\end{equation}
\end{remark}
\begin{remark}\label{r:us-hom-tvfd}
  Note that, in the previous remark, we did not need the assumption
  that the top polyvector field $\Xi_{V_s}$ on $U_s \cong V_s$ was the
  restriction of a homogeneous polyvector field on $Y$.  Assume only
  that $\Xi_{V_s}$ is nonvanishing on $V_s^\circ$.  We claim that
  there exists a homogeneous polyvector field $\Xi_Y$ on $Y$ which is
  nonvanishing on $Y^\circ := Y \setminus \{s\}$.  Since, as we
  remarked before, $H(V_s)$ does not depend on the choice of
  $\Xi_{V_s}$ nonvanishing on $V_s^\circ$ (since $H(V_s)$ does not
  change when multiplying $\Xi_{V_s}$ by any nonvanishing function),
  $M(V_s)$ is independent of this choice as well. Thus we can assume
  without loss of generality that $\Xi_{V_s}$ is the restriction of a
  homogeneous polyvector field $\Xi_Y$ on $Y$.

  To prove the claim, we first restrict to a formal (or analytic)
  completion $\hat X_s$.  Write $\Xi_{\hat X_s} := \Xi_{V_s} |_{\hat
    X_s}$ as a sum of homogeneous components, $\Xi_{\hat X_s} =
  \sum_{m=N}^\infty (\Xi_{\hat X_s})_m$.  Here $N \in \bZ$ exists
  because, for some $N \in \bZ$, one has $(T^{\dim X}_{\hat X_s})_{<N}
  = 0$: for instance, if $\omega \in \Omega^{\dim X}_{Y}$ is any
  homogeneous top differential form, of weight $-N \geq 0$, then
  $(T^{\dim X}_{\hat X_s})(\omega) \subseteq \hat \cO_{\hat X_s}$, and
  hence $(T^{\dim X}_{\hat X_s})_{<N}=0$.

  Now, it cannot be that, for all $m$, $(\Xi_{\hat X_s})_m \in
  \mathfrak{m}_s \cdot \Xi_{\hat X_s}$, for $\mathfrak{m}_s$ the
  maximal ideal of $s$, since in that case $\Xi_{\hat X_s} \in
  \mathfrak{m}_s \cdot \Xi_{\hat X_s}$, which is impossible.  For all
  $m$, write $(\Xi_{\hat X_s})_m = f_m \cdot \Xi_{\hat X_s}$.  Then
  for some $m$, $f_m$ is invertible in a neighborhood of $s$ (and we
  remark that this implies that $f_k = 0$ for $k < m$, since $f_k
  f_m^{-1}$ is regular on $Y$ of weight $k-m$).  Therefore, $f_m^{-1}
  \Xi_{\hat X_s} = (\Xi_{\hat X_s})_m$ is homogeneous.  Now,
  homogeneous elements of $T_{\hat X_s}$ are regular on all of $Y$,
  since $Y$ is conical.  Thus, $(\Xi_{\hat X_s})_m$ is the restriction
  of a regular homogeneous top polyvector field $\Xi_Y$ on $Y$ which
  is nonvanishing on $Y^\circ$, as desired.
\end{remark}
\subsection{General conjecture on  the indecomposable summand of $M(X)$}

For any complex algebraic or analytic variety $X$, let
$\Omega_X^\bullet$ denote
the commutative differential graded algebra of K\"ahler forms on
$X$. Recall that this is the algebra generated by $\cO_X$ in degree
zero and symbols $\sdf f$, for all $f \in \cO_X$, in degree one,
subject to the relations
\[
\sd (fg) = \sd(f) \cdot g + f \cdot \sd(g), \quad f,g \in \cO_X,
\]
and equipped with the differential, which we also denote abusively by
$\sd$, which is the unique derivation sending $f$ to $\sdf f$ for all
$f \in \cO_X$ and annihilating the symbols $\sdf f$ for all $f \in
\cO_X$.   


When $X$ is not necessarily an affine variety, then we let $\Omega_X^\bullet$
denote the quasicoherent sheaf of K\"ahler forms on $X$. Note that the
commutative differential graded algebra structure is compatible with
restriction, i.e., this is a sheaf of differential graded
$\cO_X$-algebras.  

For every $s \in X^\sing$,
let $U_s$ be a neighborhood of $s$ which is a complete intersection in a polydisc (which
is disjoint from $X^\sing \setminus \{s\}$).  Let $U_s^\circ := U_s \setminus \{s\}$ be the
punctured neighborhood.
\begin{definition}
  Let $\Omega^{\dim X}_\conv(U_s^\circ) \subseteq \Omega^{\dim
    X}(U_s^\circ)$ denote the vector space of holomorphic $(\dim
  X)$-forms $\alpha$ on $U_s^\circ$ which are $L^2$-convergent at $s$,
  i.e., such that $\int \alpha \wedge \bar \alpha$ is convergent in a
  neighborhood of $s$.
\end{definition}
\begin{definition}
  Let $\Omega^{\dim X}_{\log}(U_s^\circ) \subseteq \Omega^{\dim X}(U_s^\circ)$ denote the vector space of
  holomorphic $(\dim X)$-forms $\alpha$ on $U_s^\circ$ which are at
  most logarithmically divergent at $s$.  That is, for every function
  $f \in \cO_{U_s}$ vanishing at $s$, the limit $|\log
  \varepsilon|^{-1} \int_{|f| \geq \epsilon} \alpha \wedge \bar
  \alpha$ exists as $\epsilon \to 0$.
\end{definition}
We can give an alternative description of the above spaces: Let
$\widetilde {U_s} \to U_s$ be a resolution of singularities with a normal
crossing exceptional divisor $D$.  Then, $\Omega^{\dim
  X}_\conv(U_s^\circ)$ consists of those forms which extend to
holomorphic forms on $\widetilde{U_s}$, and $\Omega^{\dim
  X}_{\log}(U_s^\circ)$ consists of those forms which extend to
meromorphic forms on $\widetilde{U_s}$ with at most first-order poles (on
$D$).
\begin{conjecture}\label{c:max-ext}
$M_{\max}$ is an extension by $H^0 j_! \Omega_{X^\smth}$ of 
\begin{equation}\label{e:log-conv}
\bigoplus_s \Omega^{\dim X}_{\log}(U_s^\circ) / \Omega^{\dim X}_\conv(U_s^\circ) \otimes \delta_s.
\end{equation}
\end{conjecture}
We can give an alternative description of the space in the conjecture,
in terms of an arbitrary resolution $\widetilde{U_s} \to U_s$ with
normal crossing exceptional divisor $D$.  Let $\Omega^{\dim D}_{D,\log}$
denote the sheaf of meromorphic $\dim D$-forms on $D$ which have at
most log poles on the intersection of components of $D$, such that the
residue along intersections of components makes sense (i.e., if $D =
\bigcup_i D_i$ and $\alpha \in \Omega^{\dim D}_{\log}$, we require
that the residues of $\alpha|_{|D_i}$ and of $\alpha_{|D_j}$ along
$D_i \cap D_j$ are the same element of $\Omega^{\dim D-1}_{{D_i \cap D_j},\log}$).
\begin{proposition}\label{p:conj-rewr}\footnote{This proposition contained an error in the published version, where regular forms were erroneously used 
instead of forms with at most log poles.}
  The vector space $\Omega^{\dim X}_{\log}(U_s^\circ) / \Omega^{\dim
    X}_\conv(U_s^\circ)$ is canonically isomorphic to
  $\Gamma(D,\Omega^{\dim D}_{D,\log})$.
\end{proposition}
\begin{proof}
Consider the exact sequence of sheaves on $\widetilde{U_s}$, 
\[
0 \to \Omega^{\dim X}_{\widetilde{U_s}} \to
\Omega^{\dim X}_{\widetilde{U_s}}(D) \to \Omega^{\dim X - 1}_{D,\log} \to 0,
\]
where $\Omega^{\dim X}_{\widetilde{U_s}}(D)$ is the space of
meromorphic $(\dim X)$-forms on $\widetilde{U_s}$ with at most simple
poles on $D$ (and no poles elsewhere), and the second nontrivial map
takes the residue along the components of $D$.  The long exact
sequence on cohomology yields, in view of the comments before the
conjecture,
\[
0 \to \Omega^{\dim X}_{\conv}(U_s^\circ) \to \Omega^{\dim
  X}_{\log}(U_s^\circ) \to \Omega_{D,\log}^{\dim X - 1} \to
H^1(\widetilde{U_s}, \Omega^{\dim X}_{\widetilde{U_s}}).
\]
The proposition will follow once we show that $H^1(\widetilde{U_s},
\Omega^{\dim X}_{\widetilde{U_s}})$ is zero. This is a
consequence of the Grauert-Riemenschneider vanishing theorem. In more
detail, the Grauert-Riemenschneider vanishing theorem for proper
birational $\rho$ implies that the complex $\rho_* \Omega^{\dim
  X}_{\widetilde{U_s}}$ of quasicoherent sheaves has vanishing
cohomology sheaves in nonzero degrees. Thus 
$H^1(\widetilde{U_s}, \Omega^{\dim X}_{\widetilde{U_s}}) = H^1(U_s,
H^0 \rho_* \Omega^{\dim X}_{\widetilde{U_s}})$, where $H^0 \rho_*
\Omega^{\dim X}_{\widetilde{U_s}}$ is the underived pushforward of
$\Omega^{\dim X}_{\widetilde{U_s}}$, i.e., it is a quasicoherent sheaf (rather than
a complex of sheaves).
Since $U_s$ is Stein,
higher cohomology of any quasicoherent sheaf on $U_s$ is zero.  Thus,
$H^1(\widetilde{U_s}, \Omega^{\dim X}_{\widetilde{U_s}}) = 0$.
\end{proof}
\begin{example}\label{ex:conj-sfc}\footnote{As for the proposition, this example contained an error in the published version, which is corrected here.}
If $X$ is a surface, then $D$ is a curve.
For $D_i$ the
components of $D$,  denote their genera by $g_i$. Let $b$ be the first
Betti number of the dual graph of $D$ (the graph whose vertices are the components and with an edge between components for every intersection point).  Then we obtain
\[
\dim \Gamma(D, \Omega^1_{D,\log}) = b + \sum_i g_i,
\]
Since $\widetilde{U_s} \to U_s$ was
arbitrary, this sum must be independent of the choice of the
resolution $\widetilde{U_s}$, and this fact is well-known.
\end{example}

\subsection{Proof of the conjecture in the locally conical
  case}\label{ss:pf-conj-con}
In this subsection we prove the following result using \S \ref{s:con}.
Call $X$ \emph{locally conical} at $s$ if there exists an analytic
neighborhood $U_s$ of $s$ which is isomorphic to an analytic
neighborhood of a conical variety. Equivalently, $(U_s, s) \cong (V_s,
s)$ where $V_s \subseteq Y$ is an analytic open subset and $Y$ is
conical with vertex $s$.
\begin{proposition}\label{p:conj-con}
  Conjecture \ref{c:max-ext} is true when, for every $s \in X^\sing$,
  $X$ is locally conical at $s$.
\end{proposition}
\begin{proof}
  Since Conjecture \ref{c:max-ext} is a local statement, we can assume
  $X$ is conical with vertex $s$. Thus $X$ is affine and equipped with
  a $\bC^\times$ action with $s$ as attracting fixed point.  As
  explained in Remark \ref{r:us-hom-tvfd}, we can further assume that
  $\Xi_{U_s}$ is the restriction of a homogeneous polyvector field on
  $X$ of some weight $d$, which is therefore nonvanishing
  on $X^\circ := X \setminus \{0\}$.  
For simplicity, we first work on $X$ instead of on $U_s$.

Thus, by Theorem \ref{t:maxind-con} and Remark \ref{r:us-hat},
$M_{\max}$ is an extension by $H^0 j_! \Omega_{X^\smth}$ of
$(\bC^{\mu_s} \otimes \delta_s)_d$, where the latter is the weight $d$
part of the maximal quotient of $M(X)$ supported at $s$.  By
Proposition \ref{p:maxind-qt}, the latter is isomorphic to $(\hat
\cO_{X,s})_d \otimes \delta_s$.

Next, note that $\Xi_{X}^{-1}$ is a holomorphic $(\dim X)$-form on
$X^\circ$, and all holomorphic $(\dim X)$-forms on $X^\circ$ are of
the form $\cO_{X} \cdot \Xi_{X}^{-1}$ (as $X$ is normal).

Note that a holomorphic $(\dim X)$-form on $X$ automatically has
positive weight (at least $\dim X$), since $X$ is conical.
Furthermore, a homogeneous holomorphic $(\dim X)$-form on $X^{\circ}$
is at most logarithmically divergent if and only if it has nonnegative
weight; similarly, it is convergent if and only if it has positive
weight.  Therefore,
\[
\Omega_{\log}^{\dim X}(X^{\circ}) = (\cO_{X})_{\geq d} \cdot
\Xi_{X}^{-1}, \quad \Omega_{\conv}^{\dim X}(X^{\circ}) =
(\cO_{X})_{\geq (d+1)} \cdot \Xi_{X}^{-1}.
\]
We may conclude that $\Omega_{\log}^{\dim X}(X^{\circ}) /
\Omega_{\conv}^{\dim X}(X^{\circ}) \cong (\cO_{X})_d$.  Since $X$ is
conical, the latter is the same as $(\hat \cO_{X,s})_d$.

Now, if we work instead on $U_s$, we can use a formal (or analytic)
completion $\hat X_s$ at $s$. Then, the same argument as above shows
that
\[
\Omega_{\log}^{\dim X}(U_s^{\circ}) = ((\hat \cO_{X,s})_{\geq d} \cap
\cO_{U_s}) \cdot \Xi_{U_s}^{-1}, \quad \Omega_{\conv}^{\dim
  X}(U_s^{\circ}) = ((\hat \cO_{X,s})_{\geq (d+1)} \cap \cO_{U_s})
\cdot \Xi_{U_s}^{-1}.
\]

Consider the canonical inclusions
\begin{multline}
  ((\hat \cO_{X,s})_{\geq d} \cap \cO_{X}) / ((\hat \cO_{X,s})_{\geq
    (d+1)} \cap \cO_{X}) \to
  ((\hat \cO_{X,s})_{\geq d} \cap \cO_{U_s}) / ((\hat \cO_{X,s})_{\geq (d+1)} \cap \cO_{U_s}) \\
  \to (\hat \cO_{X,s})_{\geq d} / (\hat \cO_{X,s})_{\geq (d+1)} \cong
  (\cO_X)_d,
\end{multline}
whose composition is an isomorphism since $X$ is conical.  Hence, the
middle term is also isomorphic to $(\cO_X)_d$, so $\Omega_{\log}^{\dim
  X}(U_s^\circ) / \Omega_{\conv}^{\dim X}(U_s^\circ) \cong (\cO_X)_d
\cong (\hat \cO_{X,s})_d$ as well.
\end{proof}

\section{Recollections on Hamiltonian vector fields}
\label{s:ham-vfd}
As before, throughout the paper $X$ can be either a complex algebraic
or complex analytic variety, and $\Omega_X^\bullet$ is its
differential graded algebra of K\"ahler forms.

Here we recall from \cite[\S 3.4]{ES-dmlv} the definition of the Lie
algebra of Hamiltonian vector fields $H(X)$ when $X$ is equipped with
a top polyvector field $\Xi_X \in \wedge^{\dim X} T_X$.  In the case
when $X$ is a complete intersection surface in affine space (or in a
Calabi-Yau variety, or as in Remark \ref{r:ci-bundles}), this is the
Hamiltonian flow with respect to the Jacobian Poisson structure on
$X$. More generally, for complete intersections in Calabi-Yau varieties
of higher dimensions, we have the corresponding top polyvector field.

Associated to $\Xi_X$, as in \cite[\S 3.4]{ES-dmlv}, is a natural Lie
algebra of ``Hamiltonian'' vector fields.  Namely, to every
$(\dim X-2)$-form $\alpha \in \Omega_X^{\dim X-2}$, we associate the Hamiltonian
vector field $\xi_\alpha := i_{\sdf \alpha} \Xi_X$.
\begin{definition}\cite[\S 3.4]{ES-dmlv} \label{d:hx}
  Let $H(X) := \{\xi_\alpha \mid \alpha \in \Omega_X^{\dim X - 2}\}$.
\end{definition}
\begin{proposition}\cite[\S 3.4]{ES-dmlv} \label{p:hx-lie}
  $H(X)$ is a Lie subalgebra of the Lie algebra of vector fields under
  the Lie bracket.
\end{proposition}
Essentially, the preceding proposition is proved by observing the
identity
\[ [\xi_\alpha, \xi_\beta] = \xi_{i_{\xi_\alpha} \sdf \beta},
\]
which implies that $[H(X), H(X)] \subseteq H(X)$.

\subsection{Global generalization}
More generally, if $Y$ or $X$ is not necessarily affine, we can make
the same definition as in Definition \ref{d:hx}, where now $H(X)$
should be replaced by the \emph{presheaf} $\mathcal{H}(X)$ of
Hamiltonian vector fields. Then, Proposition \ref{p:hx-lie} carries
over to show that $\caH(X)$ is a presheaf of Lie algebras, i.e.,
$[\caH(X), \caH(X)] = \caH(X)$. See \cite[\S 2.10, Corollary
4.4]{ES-dmlv} for details.

\section{Proof of results from \S \ref{s:mr} and \S \ref{s:dmr}}
\label{s:mr-dmr-pf}
As pointed out in \S \ref{s:dmr}, all of the theorems in \S \ref{s:mr}
are implied by the more general versions in \S \ref{s:dmr}, and in
particular, everything follows from Theorems \ref{t:mt-dm-2d} and
\ref{t:mt-dm-2d-fl}, in their more general forms where $X$ is allowed
to be an analytically locally complete intersection of arbitrary
dimension $\geq 2$ equipped with a top polyvector field $\Xi_X$
vanishing at only finitely many points (i.e., at the finite singular
locus). To prove these results, we follow the steps in the outline of
the proof in \S \ref{ss:mr-proof-outline}, with subsections here
numbered the same as the outline there.

As the statements are local, it suffices to replace $X$ with an open
neighborhood of each $s \in \Xsing$, i.e., to assume $|\Xsing| =
1$. Note that if $X$ is smooth, then the theorems immediately follow
from the fact that, in this case, $M(X) \cong \Omega_{\Xsmth}$ by
\cite[Example 2.37]{ES-dmlv}).

It suffices to assume that $X$ is an analytic complete intersection in
a polydisc $Y$.  We henceforth assume this. Let $\dim X = n$ and say
that $X$ is cut out of $Y \cong B^{n+k}$ by $k \geq 1$ analytic functions,
$f_1, \ldots, f_k \in \cO_Y$.  Up to shrinking $Y$, we can assume in
fact that $X$ is contractible (\cite{Gil-eavlc}, cf.~also
\cite[2.10]{Mil-spch}).

\subsection{The smoothing $X_t$}
For generic $c_1, \ldots, c_k$, it follows that the complete
intersection cut out by $f_1-c_1, \ldots, f_k-c_k$ is smooth.
Therefore we can pick $c_1, \ldots, c_k$ such that the one-parameter
family $X_t := Z(f_1 - t c_1, \ldots, f_k - t c_k) \subseteq Y$ is
generically smooth.  As mentioned in the introduction, by
\cite{Mil-spch, Ham-ltekr}, for a small enough open ball $B(s)$ about
$s$ and for $0 < |t| \ll 1$, $X_t$ is smooth and the intersection $X_t
\cap B(s)$ is homotopic to a bouquet of $\mu_s$ spheres of dimension
$n$.  Let us shrink $Y$ so that $X_t$ is itself homotopic to a bouquet
of $\mu_s$ spheres of dimension $n$ for $t \in D \subseteq \bC$, where
$D$ is a suitably small open ball about $0$; we also assume that $X_t$
is smooth for $t \in D \setminus \{0\}$.


\subsection{The structure of the $\caD$-module $M(X)$}\label{ss:mx-str}
As explained in \cite[Example 2.37]{ES-dmlv}, $M(X^{\smth}) \cong
\Omega_{X^{\smth}}$, since $X^{\smth}$ is Calabi-Yau.  As in the
introduction, let $j: X^{\smth} \into X$ be the inclusion, so that
$M(X^{\smth}) = j^! M(X) = j^* M(X)$, which is therefore
$\Omega_{\Xsmth}$.  Then by adjunction we have a canonical morphism $N
:= H^0 j_! \Omega_{\Xsmth} = H^0 j_! j^!  M(X) \to M(X)$.  This is an
isomorphism on $\Xsmth$, and hence the cokernel, call it $K'$, is
supported on $\Xsing$.  Thus we have the exact sequence
\begin{equation}\label{e:m-seq}
  N \to M(X) \to K' \to 0.
\end{equation}
Theorem \ref{t:mt-dm-2d} reduces to showing that the first morphism,
$N \to M(X)$, is injective.  We will show this in \S \ref{ss:cp}
below.

As observed in \S \ref{ss:dm-2d} and \S \ref{ss:dm-hd}, $N$ is an
indecomposable extension of the simple $\caD$-module $\IC(X) = j_{!*}
\Omega_{\Xsmth}$, of the form
\begin{equation}
  0 \to K \to N \to \IC(X) \to 0, \quad K = \Ext^1(\IC(X), \delta_s)^* 
  \otimes \delta_s \cong H^{n - 1}(X \setminus \{s\}) \otimes \delta_s.
\end{equation}
For the final equality, we used \eqref{e:ic-ds} and the fact that $X$
is contractible.

In particular, because $N$ is indecomposable, it has no quotient
supported at $s$.  Therefore, the quotient $M(X) \onto K'$ is the maximal
quotient of $M(X)$ supported at $s$.  That is, $K' = H^0 i_* i^* M(X)$,
where $i: \{s\} \to X$ is the embedding.  We recall the structure of
this $K'$ in the next section.


\subsection{The maximal quotient of $M(X)$}\label{ss:maxquot}
Recall from \cite[Corollary 5.9]{ES-dmlv} the following formula for
the maximal quotient $K'$ of $M(X)$ supported at $s$:
\begin{equation}
K' := H^0 i_* i^* M(X) \cong \delta_s^{\mu_s}.
\end{equation}
The reason was simple: first, $K'$ identifies with $\Hom(M(X),
\delta_s)^* \otimes \delta_s \cong (\hat \cO_{X,s})_{H(\hat X_s)}
\otimes \delta_s$ (the second isomorphism is \cite[Lemma
5.10]{ES-dmlv}), and then the latter was computed in \cite[\S
5.2]{ES-dmlv} using \cite[Proposition 5.7.(iii)]{Gre-GMZ}. Here, $\hat
X_s = \Spf \hat \cO_{X,s}$ is the formal neighborhood of $s$ in $X$.

\subsection{The family $M(\mathcal{X})$ of $\caD$-modules}
Let $\mathcal{X} \subseteq Y \times D$ be the total space of the
family $X_t \subseteq Y$ for $t \in D$.  Let $i_t: X_t \into
\mathcal{X}$ and $i_{\mathcal{X}}: \mathcal{X} \into Y$ be the closed
embeddings so that $i_{\mathcal{X}} \circ i_t(X_t) = X_t \times \{t\}
\subseteq Y \times D$.  We now consider the family, call it $M(\mathcal{X})$,
of $\caD_Y$-modules over the disc $D$, whose fiber at each point
$t \in D$ is $(i_t)_* M(X_t)$. More precisely, this is a right
module over
the algebra $\caD_{Y \times D/D} = \cO_D \otimes \mathcal{D}_Y$
of relative differential operators on $Y \times D$ over $D$, i.e., of
$\cO_D$-$\mathcal{D}_Y$ modules. It can also be defined by taking the
$\caD_{Y \times D/D}$-module $\caD_{\mathcal{X}}$ whose fiber over each $t \in D$
is $\caD_{X_t}$, and quotienting by the left action of Hamiltonian vector fields.
(Note that, in the case where $X$ is two-dimensional, $\mathcal{X}$ is a Poisson
variety, and the global sections of $M(\mathcal{X})$ identify with the $\caD$-module
also denoted $M(\mathcal{X})$ defined earlier).

Since $X_t$ is smooth for $t \neq 0$, we have by
\cite[Example 2.37]{ES-dmlv} that $M(X_t) \cong (i_t)_* \Omega_{X_t}$
for $t \neq 0$. Our goal is to show that $M(\mathcal{X})$ is flat over
$D$.  Notice that, since $D \subseteq \bC$ is one-dimensional
(complex), $M(\mathcal{X})$ is flat if and only if it is torsion-free
(i.e., $(i_{\mathcal{X}})_* M(\mathcal{X})$ is torsion-free, or the
global sections of $M(\mathcal{X})$ form a torsion-free
$\cO_{D}$-module). So Theorem \ref{t:mt-dm-2d-fl} reduces to:
\begin{proposition}\label{p:tf}
  $M(\mathcal{X})$ is torsion-free over $D$.
\end{proposition}
In the remainder of the section we prove this result.  As a first
step, let $M(\mathcal{X})_\tor \subseteq M(\mathcal{X})$ denote the
$\mathcal{O}_D$-torsion.  We claim that $M(\mathcal{X})_\tor$ is a sum of
delta-function $\caD$-modules concentrated at $s$, i.e., $\delta_s$ as
a $\caD_Y$-module (with the action of $\cO_D$ factoring through the evaluation
at $t=0$).

Indeed, consider the localization, call it $M(\mathcal{X} \setminus \{s\})$,
of $M(\mathcal{X})$ to $Y \setminus \{s\}$, i.e., to a family of
$\caD_{Y \setminus \{s\}}$-modules. We claim that this family has no torsion.
 Indeed, the fiber over $t \in D$ of this
localization, call it $M(\mathcal{X} \setminus \{s\})$, 
is $\Omega_{X_t^{\smth}}$ for all $t \in D$ (which is
just $X_t$ for $t \neq 0$), which is simple. Since it is a coherent
$\cO_D \otimes \caD_{Y \setminus \{s\}}$-module (actually already a coherent
$\caD_{Y \setminus \{s\}}$-module), we can conclude that it is torsion-free
over $\cO_D$. Since $M(\mathcal{X})$ is also coherent, its $\cO_D$-torsion
must be of the form $\delta_s^r$ for some $r \geq 0$, as desired.


Let $M(\mathcal{X})' := M(\mathcal{X})/M(\mathcal{X})_{\tor}$. Then
this is torsion-free over $D$, and hence a flat family of
$\caD_Y$-modules.  Let $M(\mathcal{X})'_t$ denote the fiber at $t$;
this is $M(\mathcal{X})_t \cong (i_t)_* \Omega_{X_t}$ for $t \neq 0$.
By the preceding paragraph, we have an exact sequence
\begin{equation}\label{e:mx-tor-es}
0 \to \delta_s^r \to (i_0)_* M(X) \to M(\mathcal{X})'_0 \to 0,
\end{equation}
for some $r \geq 0$.  The proposition then reduces to showing that $r
= 0$. To do so, we will consider $\pi_* M(X_t)$.

\subsection{Euler-Poincar\'e characteristic of $\pi_* M(X_t)$}
Given a finite-dimensional $\bZ$-graded vector space $V_\bullet$, let 
$\chi(V_\bullet) := \sum_{m \in \bZ} (-1)^m \dim V_m$
denote its Euler-Poincar\'e characteristic. 

Since $M(\mathcal{X})'$ is a flat family, $\chi(\pi_*
M(\mathcal{X})'_t)$ is constant in $t$. Since $X_t$ is smooth and
homotopic to a bouquet of $\mu_s$ spheres for $t \neq 0$, we conclude
that
\begin{equation}
\chi(\pi_* M(\mathcal{X})'_0) = 
\chi(\pi_* M(X_t)) = \mu_s + (-1)^n, \quad t \neq 0.
\end{equation}
By \eqref{e:mx-tor-es}, we can rewrite this as
\begin{equation}\label{chipim+r}
\chi(\pi_* M(X)) = \mu_s + (-1)^n + r.
\end{equation}
The main step is to compute $\chi(\pi_* M(X))$ using the structure of
$M(X)$ from \S \ref{ss:mx-str}.
 First of all, we evidently have
\begin{equation}\label{e:chipim}
\chi(\pi_* M(X)) = \chi(\pi_* \IC(X)) + \mu_s + h^{n-1}(X \setminus \{s\}) - q,
\end{equation}
where $q \geq 0$ is the dimension of the kernel of the morphism $N \to
M(X)$.  The work now reduces to computing $\chi(\pi_*\IC(X))$.
\subsection{Computation of $\chi(\pi_* \IC(X))$}
The goal of this subsection is to prove
\begin{proposition}\label{p:icfla}
$\chi(\pi_* \IC(X)) = (-1)^n - h^{n-1}(\Xsmth)$.
\end{proposition}
\begin{proof}
We break this into steps:
\begin{enumerate}
\item Let $\bar X$ be the one-point compactification of $X$.  Up
to choosing $X=U_s$ a small enough neighborhood of $s$, we have
a homeomorphism $\Xsmth= U_s \setminus \{s\} \cong (0,1) \times L$ for $L$
(the link of the singularity) a manifold of real dimension $2 \dim_\bC X - 1$.
Thus $\bar X \cong X \sqcup_{X^{\smth}} X$ as a topological space.  
Moreover, $\bar X$ is homotopic to the suspension of $\Xsmth$.
\item We use
  the classical formula of \cite[\S 6.1]{GM-iht} (cf.~also, e.g.,
  \cite[(1)]{Dur-ihBn}, which is stated for algebraic varieties but
  extends to the analytic case): If $\bar X$ is a compact analytic
  variety of dimension $n$ with isolated singularities, with smooth
  locus $\bar X^{\smth}$:
\begin{equation}\label{e:gm-fla}
\operatorname{IH}_i(\bar X) := \pi_{n - i} \IC(\bar X) = 
\begin{cases}
H_i(\bar X), & \text{if $i > n$}, \\
\operatorname{Im}(H_n(\bar X^\smth) \to H_n(\bar X)), & \text{if $i = n$}, \\
H_i(\bar X^\smth), & \text{if $i < n$}.
\end{cases}
\end{equation}
\item 
  We apply the Mayer-Vietoris sequence to $\bar X = X \sqcup_{\Xsmth}
  X$ to compute $\pi_* \IC(\bar X)$ in terms of $\pi_* \IC(X)$
  and $\pi_* \IC(\Xsmth) = H_{n-*}(\Xsmth)$:
\begin{equation}\label{e:ihh-diff}
  \chi(\pi_* \IC(\bar X)) = 2 \chi(\pi_* IC(X)) - (-1)^n \chi(\Xsmth),
\end{equation}
where $\chi(\Xsmth) = \chi(H_*(\Xsmth))$ is the Euler characteristic
of $\Xsmth$.
\item Now apply \eqref{e:gm-fla} to the LHS of
  \eqref{e:ihh-diff}. Note that $H_n(\bar X^\smth) \to H_n(\bar X)$ is
  zero since it factors through $H_n(X)$, which is zero as $n =
  \dim_{\bC} X > 0$ and $X$ is contractible.  We conclude that
\begin{equation}
  \sum_{i=0}^{n-1} (-1)^{n-i} h^i(\Xsmth) + \sum_{i=n+1}^{2n} (-1)^i h^i(\bar X)
  =
  2 \chi(\pi_* \IC(X))
  - \sum_{i=0}^{2n} (-1)^{n-i} h^i(\Xsmth).
\end{equation}
\item Using that $\bar X$ is homotopic to the suspension of $\Xsmth$
  (hence $H_i(\bar X) = H_{i-1}(\Xsmth)$ for $i \geq 2$), and that
  $H_{2n}(\Xsmth) = 0$ as $\Xsmth$ is a noncompact real $2n$-manifold,
  the above simplifies to
\begin{equation}
  \chi(\pi_*\IC(X)) = \sum_{i=0}^{n-1} (-1)^{n-i} h^i(\Xsmth).
\end{equation}
\item Finally, we apply the fact that, for $X$ an analytically locally
  complete intersection of dimension $n$ with an isolated singularity,
  the link of the singularity is $(n-2)$-connected (see
  \cite{Mil-spch} and \cite[Korollar 1.3]{Ham-ltekr}). Therefore
  $H_{i}(\Xsmth) = 0$ for $1 \leq i \leq n-2$.  We obtain the
  proposition. \qedhere
\end{enumerate}
\end{proof}

\subsection{Proof of Theorems \ref{t:mt-dm-2d} and
  \ref{t:mt-dm-2d-fl}}\label{ss:cp}
Putting together Proposition \ref{p:icfla} and \eqref{e:chipim}, we obtain
\begin{equation}\label{e:chipim-q}
\chi(\pi_* M(X)) = (-1)^n + \mu_s - q.
\end{equation}
On the other hand, comparing this with \eqref{chipim+r} and the fact
that $q, r \geq 0$, we obtain
\begin{equation}
  q=r=0.
\end{equation}
Since $r=0$, this completes the proof of Proposition \ref{p:tf}, as
remarked there, and hence also Theorem \ref{t:mt-dm-2d-fl}, as also
pointed out there.  Since $q=0$, by the definition of $q$ in
\eqref{e:chipim} and the comment after \eqref{e:m-seq}, Theorem
\ref{t:mt-dm-2d} is proved as well.

\section{Proof of Theorem \ref{t:maxind-con} and Proposition
  \ref{p:maxind-qt} for cones over smooth curves in
  $\bP^2$}\label{s:mmax-pf-2d}
For concreteness, we first prove these results in the case that
$X\subseteq \bC^3$ is the cone over a smooth curve in $\bP^2$ (with
vertex $s=0$), even though the general proof is essentially the same
(and for most of it, we will copy and adapt the proof given here).  We
assume that the Poisson bivector $\Xi_X$ has weight $d$, and hence
that $X$ is the zero locus of a homogeneous polynomial in $\bC^3$ of
weight $d+3$.

We prove Proposition \ref{p:maxind-qt} first, and use it in
the proof of Theorem \ref{t:maxind-con}.  

\subsection{Proof of Proposition \ref{p:maxind-qt}}
As recalled in \S \ref{ss:maxquot}, the maximal quotient of $M(X)$
supported at $0$ is canonically identified with $(\hat
\cO_{X,0})_{H(\hat X_0)}$, which is just $\HP_0(\cO_{X})$ since
$\cO_{X}$ is positively graded. Now, Hamiltonian vector fields all
have weight at least $d+1$.  As a result, $\HP_0(\cO_{X})= (\cO_X)_m$
for all $m \leq d$.

\subsection{Proof of Theorem
  \ref{t:maxind-con}}\label{ss:maxind-con-pf-2d}
For a contradiction, suppose that there were a direct sum
decomposition
\begin{equation}\label{e:mxd-decomp-2d}
M(X)_d = M(X)_d' \oplus \delta_0.
\end{equation}
This would induce a decomposition on solutions valued in every
$\caD$-module $N$ on $X$,
\begin{equation}\label{e:mxd-sol-decomp-2d}
  \Hom(M(X)_d, N) \cong \Hom(M(X)_d', N) \oplus \Hom(\delta_0, N).
\end{equation}
Let $N$ be the space of smooth, compactly supported distributions on
$X$, i.e., smooth, compactly supported distributions on $\bC^3$
scheme-theoretically supported on $X$.  Let $w: \delta_0 \to N$ be the
canonical inclusion of the delta function $\caD$-module, i.e., $w(1)
\in N$ is the delta function distribution, where $1 \in \delta_0$ is
the canonical generator.  Let $\Eu$ denote the holomorphic Euler
vector field on $X$.  We will prove the following result.  Let $1 \in
M(X)$ be the canonical generator.  Note that $\Eu$ acts as an
endomorphism of $M(X)$, $1 \cdot \Phi \mapsto 1 \cdot \Eu \cdot \Phi$
for all $\Phi \in \caD_X$, so let $T_{\Eu}: M(X) \to M(X)$ denote this
endomorphism.
\begin{lemma}\label{l:maxind-con-phi-2d}
  There is a map $\Phi: \Hom(M(X)_d,\delta_0) \to \Hom(M(X)_d, N)$ such
  that $\Phi(v)\circ (T_{\Eu}-d) = w \circ v$ for all $v \in
  \Hom(M(X)_d,\delta_0)$.
\end{lemma}
We prove the lemma in \S \ref{ss:l-maxind-con-phi-pf-2d} below.  Using
the lemma, we conclude the theorem as follows.  First, note that
$\ad(T_{\Eu})$ acts semisimply on $\End(M(X))$, since $\ad(\Eu)$ acts
semisimply on global sections of $\caD_X$, and $\End(M(X))$ is a
homogeneous subquotient thereof: in more detail,
\[
\End(M(X)) =
(\Gamma(X,H(X) \cdot \caD_X) \setminus \Gamma(X,\caD_X))^{H(X)}.
\]
Moreover, $\ad(T_{\Eu})$ acts with eigenvalue $m-k$ on $\Hom(M(X)_k,
M(X)_m)$ for all $m,k$, and hence it acts there by $(m-k) \cdot \Id$.
In particular, $T_{\Eu}$ is central in $\End(M(X)_m)$ for all $m$.
Thus, $T_{\Eu}$ preserves all direct summands of $M(X)_m$ for all $m$.

Suppose now that $M(X)_d$ had a nonzero direct summand, $K$, supported
at the origin.  Then, $T_{\Eu}$ preserves $K$. Since $\Eu$ acts
semisimply on global sections of any $\caD$-module supported at the
origin, it follows that $T_{\Eu}$ is a semisimple endomorphism of
$K$. Since $T_{\Eu}-d$ is a nilpotent endomorphism of $M(X)_d$, it
follows that $T_{\Eu}-d$ restricts to the zero endomorphism of $K$.

Now assume that $K \cong \delta_0$, up to taking a further summand.
Let $v: M(X)_d \onto \delta_0$ be the corresponding projection. Then
Lemma \ref{l:maxind-con-phi-2d} implies that $\Phi(v) \circ
(T_{\Eu}-d)$ is nonzero and factors through $v$, and in particular,
$\Phi(v) \circ (T_{\Eu}-d)$ does not vanish on the summand $K$.  This
contradicts the fact that $T_{\Eu}-d$ restricts to zero on $K$.  This
proves the theorem.

\subsection{Proof of Lemma
  \ref{l:maxind-con-phi-2d}}\label{ss:l-maxind-con-phi-pf-2d}
Let $\omega := \Xi_{X}^{-1}$ be the meromorphic volume
form on $X^\smth$ (which is holomorphic on $X^\smth$).  Then $\omega$
has weight $-d$.

Let $S_t := \{x \in X \mid |x| = t\} \subseteq \bC^3$ be the
intersection of $X$ with the (five-dimensional) sphere of radius $t$,
and $B_t := \{x \in X \mid |x| \leq t\}$ the corresponding closed
balls. Then $S_t$ and $B_t$ are compact for all $t \in \bR_{\geq 0}$.

For every $Q \in (\cO_{X})_d$, consider the partially defined
functional on $C_c^{\infty}(X)$:
\[
\phi_Q: \alpha \mapsto \int_{X} \alpha \omega \wedge \bar Q \bar
\omega.
\]
For all $m \geq 0$, let $(C_c^\infty(X))_{> m}$ be the subspace of
(smooth compactly-supported) functions all of whose derivatives up to
and including order $m$ vanish.
\begin{lemma}\label{l:phiq-conv}
The functional $\phi_Q$ converges for $\alpha \in (C_c^\infty)_{>d}$.
\end{lemma}
\begin{proof} Let $r: \bC^3 \to \bR_{\geq 0}$ be the radial function
  $r(z)=|z|$.  We can rewrite
\[
\phi_Q(\alpha) = \int_{0}^\infty dt \int_{S_t} \alpha \bar Q
(\omega \wedge \bar \omega)/dr.
\]
Then the above integral converges absolutely, since for $C_\alpha > 0$
such that $|\alpha(z)| < C_\alpha |z|^{d+1}$, and all $t > 0$,
\[
|\int_{S_t} \alpha \bar Q \omega \wedge \bar \omega/dr| < C_\alpha
|\int_{S_t} |z|^{d+1} |Q| \omega \wedge \bar \omega/dr| = C_\alpha
|\int_{S_1} |z|^{d+1} |Q| \omega \wedge \bar \omega/dr|.
\]
Letting $C$ equal the right-hand side and $R > 0$ be such that
$\alpha$ is supported in $B_R$, we obtain $|\phi_Q(\alpha)| < C \cdot
R$, which proves the absolute convergence.
\end{proof}
Next, extend $\phi_Q$ arbitrarily to a functional on all of
$C_\infty^c(X)$. Note that the difference between any two such
extensions annihilates $C_c^{\infty}(X)_{>d}$, so is a linear
combination of derivatives of the delta distribution at $0$ of orders
$\leq d$.  We claim that $\phi_Q$ is annihilated by every Hamiltonian
vector field $\xi$ on $X$, i.e., $\phi_Q \in \Hom(M(X), N)$.  It
suffices to let $\xi$ be homogeneous, say of weight $m$. Since $\Xi$
has weight $d$, it follows that $m > d$.

First, note that $\phi_Q \cdot \xi$ is supported at the origin, since
$\omega$ is invariant under Hamiltonian flow on $X^\smth$.  Now,
$\epsilon := \phi_Q \cdot(\Eu-d)$ is supported at the
origin. Moreover, it annihilates $C_c^{\infty}(X)_{>d}$, and hence
$\epsilon$ is a sum of homogeneous distributions supported at $0$ of
weights $\geq -d$. Hence $\epsilon \cdot \xi$ is a sum of
distributions supported at $0$ of weights $\geq m-d > 0$, and hence is
zero.  Thus
\[
0 = \phi_Q \cdot (\Eu-d) \cdot \xi = (\phi_Q \cdot \xi) \cdot (\Eu+m-d).
\]
Since $m-d > 0$ and all distributions supported at $0$ are linear
combinations of distributions of nonpositive weights, it follows that
$\phi_Q \cdot \xi = 0$, as desired.

We also saw above that $\phi_Q \cdot (\Eu-d)$ is supported at the
origin and annihilates $C_c^{\infty}(X)_{>d}$. Up to our choice of
$\phi_Q$, i.e., adding a linear combination of derivatives of the
delta function distribution at $0$ of weights $\geq -d$, we can assume
that $\phi_Q \cdot (\Eu-d)$ has weight $-d$, and hence $\phi_Q \cdot
(\Eu-d)^2 = 0$.  Thus, $\phi_Q \in \Hom(M(X)_d, N)$. Note that this
uniquely determines $\phi_Q$ up to an element of $(\cO_X)_d^* \cong
\Hom(M(X)_d, \delta)$.  By picking a basis of $(\cO_X)_d$, we can
extend the assignment $Q \mapsto \phi_Q$ to a linear map $(\cO_X)_d
\to \Hom(M(X)_d, N)$, and any two such maps differ by a linear map
valued in $\Hom(M(X)_d, \delta)$.

Consider next the Hermitian pairing on $\cO_X$,
\begin{equation}
  \langle P, Q \rangle := \int_{S_1} P \bar Q \omega \wedge \bar \omega / dr,
\end{equation}
which restricts to nondegenerate pairings $(\cO_X)_m \otimes (\cO_X)_m
\to \bC$ for all $m \geq 0$.  We obtain an antilinear isomorphism
$\iota: \Hom(M(X)_d, N) = (\cO_X)_d^*\iso (\cO_X)_d$, so that $\langle
P, \iota(v)\rangle = v(P)$.  Composing this with the linear map $Q
\mapsto \phi_Q$ above, we obtain a linear map $\Phi: \Hom(M(X)_d,
\delta) \to \Hom(M(X)_d, N)$. Any two such choices of $\Phi$ differ by
an element of $\Hom(M(X)_d, \delta)$.

We now claim that for every such $\Phi$, Lemma
\ref{l:maxind-con-phi-2d} is satisfied, up to rescaling $\Phi$. Since
$\Hom(M(X)_d, \delta)$ is annihilated by $T_{\Eu}-d$ (i.e., all
solutions of $M(X)_d$ supported at $0$ are annihilated by $\Eu-d$), it
suffices to show this for any particular $\Phi$.  By our definition of
$\Phi$, we need to prove that the following holds up to a constant
factor:
\begin{equation}
  \phi_{Q} \cdot (\Eu-d) \cdot P = 
  \langle P, Q \rangle w(1), \quad \forall P \in (\cO_X)_d.
\end{equation}

To prove this, we construct a particular $\phi_Q$ as follows.  Let
$\beta \in C_c^{\infty}(X)_{> d}$ be any function such that
$\beta(0)=1$, and assume it is supported in the unit ball $B_1$.  Let
$\beta_q$ be the function $\beta_q(x) := \beta(q^{-1} \cdot x)$, which
is supported in $B_q$.  For all $q > 0$, consider the projection to
$C_c^\infty(X)_{> d}$ along $\bigl( (\cO_X)_{\leq d} \otimes
(\overline{\cO_X})_{\leq d} \bigr) \cdot \beta_q$,
\[
\pr^q_{> d}: C_c^{\infty}(X) \to C_c^\infty(X)_{> d}.
\]
Then, for all $q$, we extend
$\phi_Q$ to the functional
\[
\phi_{Q,q} := \phi_Q \circ \pr^q_{> d}.
\]

Let $\epsilon := \phi_{Q,q} \cdot (\Eu-d) \in \Hom(M(X)_d, \delta)$,
which does not depend on $q$.  Lemma \ref{l:maxind-con-phi-2d} follows
from the following result (once we rescale $\Phi$ by $-2$):
\begin{lemma} \label{l:phiq-2d}
For all $P \in (\cO_X)_d$, 
\begin{equation}
\epsilon \cdot P = -
\frac{1}{2} \langle P, Q \rangle w(1).
\end{equation}
\end{lemma}
\begin{proof}
  Since $\epsilon \cdot (\Eu-d) \cdot P$ is a multiple of $w$ for all
  $P \in (\cO_X)_d$, it is enough to show the identity after
  evaluating on a single function $H \in C_c^\infty(X)$ with $H(0)
  \neq 0$.  Let $h \in C_c^\infty(\bR)$ be a function such that
  $h(0)\neq 0$, and let $H(x) := h(|x|^2)$ be the corresponding
  spherically symmetric function on $X$. Then,
\[
\epsilon (P\cdot H) = \int_{X} \bar Q P \Eu \pr^q_{> d}(H) \omega
\wedge \bar \omega,
\]
for all choices of $q$. Taking the limit as $q \to 0$, this becomes
\[
\int_{0}^\infty dt \cdot t \cdot h'(t^2) \cdot \langle P, Q \rangle =
- \frac{1}{2} \langle P, Q \rangle h(0) = -\frac{1}{2} \langle P, Q
\rangle H(0). \qedhere
\]
\end{proof}

\section{Proof of Theorem \ref{t:maxind-con} and Proposition
  \ref{p:maxind-qt} in generality}\label{s:mmax-pf}

\subsection{Proof of Proposition \ref{p:maxind-qt}}
As recalled in \S \ref{ss:maxquot}, the maximal quotient of $M(X)$
supported at $s$ is canonically identified with $(\hat
\cO_{X,s})_{H(\hat X_s)}$. Now, $H(\hat X_s)$ is obtained by
contracting $\Xi$ with differential $(n-2)$-forms on the formal
neighborhood $\hat X_s$.  Since $\hat X_s$ is conical by assumption,
differential $(n-2)$-forms are convergent sums of homogeneous forms of
positive weight.  Therefore, $H(\hat X_s)_m = 0$ for $m \leq d$.  As a
result, $((\hat \cO_{X,s})_{H(\hat X_s)})_m = (\hat \cO_{X,s})_m$ for
all $m \leq d$.

\subsection{Proof of Theorem \ref{t:maxind-con}}
We begin as in \S \ref{ss:maxind-con-pf-2d}. For convenience in
referring to that section, we let $0:=s$, i.e., consider $s$ to be the
origin of our conical variety $X$.  Assume for a contradiction that we
have a decomposition \eqref{e:mxd-decomp-2d}, which induces the
decomposition \eqref{e:mxd-sol-decomp-2d} on solutions valued in
$\caD$-modules $N$.

As before, let $N$ be the space of smooth, compactly supported
distributions on $X$ (since $X$ is conical, it embeds into an affine
space, and we can define this space as the smooth, compactly supported
distributions on the ambient space which are scheme-theoretically
supported at $X$, i.e., annihilate the polynomial functions vanishing
on $X$; this defines $N$ independently of the embedding up to
canonical isomorphism).  Just as before, $w: \delta_0 \to N$ denotes
the canonical inclusion of the delta function $\caD$-module, $\Eu$
denotes the holomorphic Euler vector field on $X$, and $1 \in M(X)$
denotes the canonical generator. Also, as before, $\Eu$ induces an
endomorphism $T_{\Eu}: M(X) \to M(X)$.

Below we will prove that Lemma \ref{l:maxind-con-phi-2d}
extends to this setting.  Then the theorem
follows from the lemma just as before.

\subsection{Proof of Lemma \ref{l:maxind-con-phi-2d} in generality}\label{ss:l-maxind-con-phi-pf}
Let $\omega := \Xi_{X}^{-1}$ be the meromorphic volume form on
$X^\smth$ (which is holomorphic on $X^\smth$).  Then $\omega$ has
weight $-d$.

Let us assume that $X$ is embedded into an affine space $\bA$ with
homogeneous (positive integral weight) coordinate functions. This can
be done, for example, by taking sufficiently many general homogeneous
functions $x_i \in \cO_X$, of weights $a_i \geq 1$.  Let $a$ be the
least common multiple of the weights $a_i$ of the $x_i$. Then we may
define a radial function $r \in C^\infty(X^{\smth})$ by $r :=
\bigl(\sum_i |x_i|^{2a/a_i}\bigr)^{1/2a}$, which extends continuously
to $X$ via $r(0)=0$, and is smooth on $X^{\smth}$. Moreover, $r^{2a}$
is a smooth function on $X$.

Let $S_t := \{x \in X \mid r(x) = t\}$ be the corresponding spheres of
radius $t$, and $B_t := \{x \in X \mid r(x) \leq t\}$ the
corresponding closed balls. Then  $S_t$ and $B_t$ are compact for all $t
\in \bR_{\geq 0}$.

As in \S \ref{ss:l-maxind-con-phi-pf-2d}, we consider the partially defined functional
$\phi_Q$, defined by the same formula.
For all $a \in \bR_{\geq 0}$, let $(C_c^\infty(X))_{> a}$ be the
subspace of (smooth compactly-supported) functions $\alpha$ such that
$\lim_{r \to 0} |\alpha/r^a| = 0$. We can restate this as follows in
terms of the derivatives of $\alpha$.  For each coordinate function
$x_i$ on $X$, let $d_i$ be its weight, and assign $\partial_i$ weight
$-d_i$.  Then $C_\infty^c(X)_{> a} \subseteq C_\infty^c(X)$ is the
subspace of smooth functions represented by smooth functions on the
affine space $\bA$ all whose derivatives of weights $\geq -a$ vanish
at the origin (this includes functions for which all derivatives
up to order $a$ vanish at the origin). In particular, for $\alpha \in
C_\infty^c(X)_{> a}$, it follows that there exists $C_\alpha > 0$ such
that $|\alpha| < C_\alpha \cdot r^{a+1}$ (and the converse holds as
well).

With these definitions, 
Lemma \ref{l:phiq-conv} extends to this context, with the same
proof.

As before, extend $\phi_Q$ arbitrarily to a functional on all of
$C_\infty^c(X)$. Note that the difference between any two such
extensions annihilates $C_c^{\infty}(X)_{>d}$, so is a linear
combination of derivatives of the delta distribution at the origin of weights
$\geq -d$.  It then follows exactly as before that $\phi_Q$ is annihilated by
every Hamiltonian vector field $\xi$ on $X$,
i.e., $\phi_Q \in \Hom(M(X), N)$.  We also can define the linear map $\Phi$ 
and the Hermitian pairing $\langle -, -\rangle$ on $\cO_X$ just
as before. Then, we claim that Lemma \ref{l:phiq-2d} extends to this setting.

The proof of Lemma \ref{l:phiq-2d} is the same as before, except that we
have to modify the function $H$ as follows.

Recall from above that $a$ was the least common multiple of the
weights $a_i$ of the coordinate functions $x_i$ which realize the embedding of
$X$ into affine space. Let $h \in
C_c^\infty(\bR)$ be a function such that $h(0) \neq 0$, and let $H \in C_c^\infty(X)$
be the function $H=h \circ r^{2a}$.
As before,
\[
\epsilon (P \cdot H) = \int_{X} \bar Q P \Eu \pr^q_{> d}(H) \omega \wedge \bar \omega,
\]
for all choices of $q$. Taking the limit as $q \to 0$, this becomes
\[
\int_{0}^\infty dt (at^{2a-1} h'(t^{2a})) \cdot \langle P, Q \rangle
= - \frac{1}{2} \langle P, Q \rangle h(0) = -\frac{1}{2} \langle P, Q \rangle H(0). 
\]
This proves Lemma \ref{l:phiq-2d} in the general setting, and hence
Lemma \ref{l:maxind-con-phi-2d}. The theorem is proved.

\bibliographystyle{amsalpha}
\bibliography{master}

\end{document}